\documentclass{amsart}

\usepackage{bbm, amsmath,amsthm,amsfonts,amssymb,bm,graphicx,mathrsfs, dsfont, mathtools, mathdots}

\usepackage{longtable}
\usepackage{manfnt}
 \usepackage{amsaddr}

\usepackage{tikz-cd}
\usepackage{adjustbox}

\usepackage{comment}
\usepackage%[backref=page]
{hyperref}

\hypersetup{colorlinks=true,citecolor=blue,linkcolor=blue,urlcolor=blue,
pdfstartview=FitH }

\theoremstyle{plain}
\newtheorem{lemma}{Lemma}
\newtheorem{theorem}[lemma]{Theorem}

\newtheorem{prop}[lemma]{Proposition}

\numberwithin{lemma}{section}
\numberwithin{equation}{section}

\theoremstyle{definition}

\theoremstyle{remark}
\newtheorem{remark}{Remark}[section]
\newtheorem{example}[remark]{Example}

\mathtoolsset{showonlyrefs=true}

\newcommand{\op}{{\rm op}}

\makeatletter
\def\Ddots{\mathinner{\mkern1mu\raise\p@
\vbox{\kern7\p@\hbox{.}}\mkern2mu
\raise4\p@\hbox{.}\mkern2mu\raise7\p@\hbox{.}\mkern1mu}}
\makeatother

\begin{document}
\author{Edgar Assing}

\address{Mathematisches Institut, Endenicher Allee 60, 53115 Bonn}

\email{assing@math.uni-bonn.de}

\title[A Density Theorem for Borel-type congruence Subgroups]{A Density Theorem for Borel-type congruence Subgroups and arithmetic Applications}

\begin{abstract}
We use a (pre)-Kuznetsov type formula to prove a density theorem for the Borel-type congruence subgroup of ${\rm GL}_n$. This has some arithmetic applications to optimal lifting and counting considered earlier by A. Kamber and H. Lavner for ${\rm GL}_3$. 	
\end{abstract}

\thanks{The author is supported by the Germany Excellence Strategy grant EXC-2047/1-390685813 and also partially funded by the Deutsche Forschungsgemeinschaft (DFG, German Research Foundation) -- Project-ID 491392403 -- TRR 358.}

\setcounter{tocdepth}{2}  \maketitle 

\maketitle

%\noindent\textbf{Statements and Declarations:} 
%\begin{itemize}
%	\item \textit{Conflicts of interest statement: }The author has no relevant financial or non-financial interests to disclose.
%	\item \textit{Data Availability statement: }Data sharing not applicable to this article as no datasets were generated or analysed during the current study.
%\end{itemize}

\section{Introduction}

One of the central problems in the theory of automorphic forms is the generalized Ramanujan Conjecture (GRC). For ${\rm GL}_n$ it predicts that all cuspidal automorphic representations are tempered. Already in the smallest interesting case, namely $n=2$, this is far from being solved. Good surveys on the topic are \cite{BB, Sa4}. A task more tractable than proving the full conjecture, which is also sufficient for many arithmetic applications, is to show that there are few exceptions to the GRC. When appropriately formalized this leads to Sarnak's density hypothesis. In this note we continue the systematic study of Sarnak's density hypothesis in higher rank. More precisely, we establish a new density theorem for Borel-type congruence subgroups. On the way the method from the prequels \cite{Bl, AB} is enhanced by a new key idea. 

For a congruence lattice $\Gamma\subseteq {\rm SL}_n(\mathbb{Z})$, let $\mathcal{F}_{\Gamma}$ be a maximal orthogonal collection of Hecke-Maa\ss\  cusp forms $\varpi\in L^2(\Gamma\backslash {\rm SL}_n(\mathbb{R})/{\rm SO}_n(\mathbb{R}))$. To each element $\varpi\in \mathcal{F}_{\Gamma}$ we associate the spectral parameter $\mu_{\varpi}=(\mu_{\varpi}(1),\ldots,\mu_{\varpi}(n))\in \mathbb{C}^n$ and define
\begin{equation}
	\sigma_{\varpi} = \max_i\vert\Re(\mu_{\varpi}(i))\vert.\nonumber
\end{equation}
We normalize the spectral parameter so that $\varpi$ is tempered at infinity if and only if $\sigma_{\varpi}=0$. The constant function $\mathbf{1}$, which is of course not cuspidal, has parameter $(\frac{1-n}{2},\frac{3-n}{2},\ldots,\frac{n-3}{2},\frac{n-1}{2})$. In particular $\sigma_{\mathbf{1}}=\frac{n-1}{2}$. We write $\mathcal{F}_{\Gamma}(M)$ for the finite sub-collection with spectral parameter bounded by $M\in \mathbb{R}_{>0}$. The (spherical) density hypothesis can now be formulated as
\begin{equation}
	N_{\Gamma}(\sigma;M) = \sharp \{\varpi\in \mathcal{F}_{\Gamma}(M)\colon \sigma_{\varpi}\geq \sigma\} \ll_{\epsilon,n} \sharp \mathcal{F}_{\Gamma}(M)^{1-\frac{2\sigma}{n-1}+\epsilon}.\label{eq:strong}
\end{equation}
This is a linear interpolation between the two extreme cases
\begin{equation}
	N_{\Gamma}(0;M) = \sharp \mathcal{F}_{\Gamma}(M) \text{ and }N_{\Gamma}(\frac{n-1}{2};M)+\sharp\{\mathbf{1}\}=1,\nonumber
\end{equation}
so that the bound \eqref{eq:strong} is analogous to a convexity bound in analytic number theory.

Achieving \eqref{eq:strong} with the required uniformity in $M$ is a formidable problem which we can not solve at present. Therefore, we fix $M$ and focus solely on the volume aspect. In view of the bound $\sharp \mathcal{F}_{\Gamma}(M)\ll_M [{\rm SL}_n(\mathbb{Z})\colon \Gamma]$, given for example in \cite{Do}, we aim for the estimate
\begin{equation}
	N_{\Gamma}(\sigma;M) \ll_{M,n,\epsilon} [{\rm SL}_n(\mathbb{Z})\colon \Gamma]^{1-\frac{2\sigma}{n-1}+\epsilon}. \label{eq:weak}
\end{equation}
This is now amenable to current technology when $\Gamma$ varies in a suitable family with growing index.

For $n=2$ (or even more generally rank one) such density theorems are fairly well understood and we refer to \cite{Hum, Hux, Iw, Sa1} and the reference within. In higher rank the situation is much more complicated. Indeed, when the problem is set up as above for certain level families, the first major breakthrough in arbitrary rank was achieved in \cite{Bl} where an improved (i.e. sub-convex) density theorem was shown for the standard (Hecke-type) congruence subgroup $\Gamma_0(q)$ of ${\rm GL}_n$. Earlier, in \cite{BBM} the ${\rm GL}_3$-case was treated. Related but weaker density estimates also appear in \cite[Section~9]{BrM}. For the principal congruence subgroup $\Gamma(q)$ Sarnak's density hypothesis was established in \cite{AB, ABN}. It should be noted that spectral families have been studied in \cite{BBR, MT, Ja}. 

Of course Sarnak's density hypothesis can be formulated for more general reductive groups than ${\rm GL}_n$. However, in general the existence of non-tempered cuspidal representations is known, so that the formulation of the GRC is slightly more complicated. Nonetheless, these exceptions are (expected to be) of low density, so that one expects the density hypothesis to hold as formulated in \eqref{eq:strong}. This has been studied in \cite{A, Man} for certain families of congruence lattices in ${\rm PGSp}_4$.\footnote{The arguments in \cite{Man} closely follow \cite{Bl} using the Kuznetsov formula. Thus, their results only include generic cuspidal representations and non-tempered cusp forms are not included by default. Of course one could account for them manually. This has been carried out in \cite{A} for the paramodular group.}

In this note we consider the family of Borel-type congruence subgroups: 
\begin{equation}
	\Gamma_2(q) = \left[\begin{matrix} \mathbb{Z} & \mathbb{Z}& \cdots & \mathbb{Z}\\ q\mathbb{Z} & \ddots &\ddots & \vdots \\ \vdots &\ddots & \ddots &\mathbb{Z}\\ q \mathbb{Z} & \cdots & q\mathbb{Z} & \mathbb{Z}  \end{matrix} \right]\cap {\rm SL}_n(\mathbb{Z}).\label{lattice}
\end{equation}
The main result is:
\begin{theorem}
For $q=l$ prime and $\sigma\geq 0$ we have
\begin{equation}
	N_{\Gamma_2(q)}(\sigma;M) \ll_{n, \varepsilon}M^K[{\rm SL}_n(\Bbb{Z}) : \Gamma_2(q)]^{1 - \frac{2\sigma}{n-1} + \varepsilon}, \nonumber
\end{equation}
for an absolute constant $K$ depending only on $n$.
\end{theorem} 

This establishes Sarnak's density hypothesis as stated in \eqref{eq:weak} with polynomial dependence on $M$. The latter is important for arithmetic applications. An alternative statement in the language of automorphic representation is given in Theorem~\ref{th_dens} below. 

\begin{remark}
The restriction to $q$ to a prime is crucial because we exploit at several places, that $\Gamma_2(q)$ gives rise to a Iwahori subgroup at the place corresponding to $q$. With a little bit of work it is probably possible to treat squarefree $q$, but the general case would require some new ideas.
\end{remark}

Following the lead of \cite{Bl} and \cite{AB} we use a (pre)-Kuznetsov formula to derive the density estimate. This produces two key problems. First, one has to give an appropriate lower bound for the spectral side. This involves showing that the first Fourier coefficient of sufficiently many $\varpi$ is not to small.  Second, the geometric side needs to be  bounded from above. This is usually done by estimating the Kloosterman sums trivially, requiring a good understanding of the ramified Kloosterman sets. The novelty in our approach here is that we do not treat these Kloosterman sets individually but on average. As a result we only need to understand the associated Kloosterman-set zeta function. Borrowing a trick from \cite{DR} translates this in the study of intertwining operators acting on unramified principal series. Studying these is still not an easy task in general, but we can exploit the combinatorial structure of the Iwahori Hecke algebra to solve this problem. A more detailed sketch of the proof can be found in Section~\ref{sec2} below.

Using the Kuznetsov formula to study density theorems is a two sided sword. Indeed, the Kuznetsov formula is a relative trace formula involving Fourier coefficients (i.e. Whittaker periods) on the spectral side. It can thus only sees generic constituents of the spectrum (i.e. those with non-zero Whittaker functions). A consequence of this is that one only picks up the generic cuspidal part of the discrete spectrum. This is the reason why we have formulated the density hypothesis only for the cuspidal part. For ${\rm GL}_n$ this is not really a restriction because all cusp forms are generic and the residual spectrum is well understood by \cite{MW2} and can be added by hand. See \cite[Section~7]{AB} where this is done for the principal congruence subgroup. In general, accounting for non-generic parts of the discrete spectrum can be quite challenging. However, the same feature that we just identified as a caveat also makes the Kuznetsov formula a powerful tool in studying the density hypothesis. Indeed, the absence of the residual spectrum from the spectral side also takes care of the constant function $\mathbf{1}$, which is somehow the arch enemy when proving density theorems. 

As usual the density theorem can be applied to derive certain counting and lifting results, see \cite{GK} for a nice survey. In the situation at hand these were already studied for ${\rm SL}_3$ in \cite{KL}. Injecting our density result in the machinery developed in \cite{JK} generalizes the work of \cite{KL} unconditionally to ${\rm SL}_n$. Precise statements can be found in Theorem~\ref{th:counting} and Theorem~\ref{th:lifting} below. \\

\textbf{Acknowledgement:} I would like to thank V. Blomer for many encouraging discussions, which led to some of the key ideas used in the paper. Further, I would like to thank S. Dawydiak for many interesting discussions concerning more abstracts aspects of Iwahori Hecke Algebras. Finally, I thank the referees for their many helpful and insightful comments.

\section{Sketch of the Proof}\label{sec2}

We will now briefly explain the machinery behind the proof of our main theorem. Doing so requires us to use some notation which will be introduced only below. We hope that this paragraph is instructive anyway. Here we stick to the situation $\Gamma=\Gamma_2(q)$ adding only some remarks concerning the lattices $\Gamma_0(q)$ and $\Gamma(q)$. Note that the same argument works more generally. But minor modifications are necessary when $\Gamma\cap U(\mathbb{R})\neq U(\mathbb{Z})$. 

The density theorem will follow from the estimate,
\begin{equation}
\sum_{\varpi\in \mathcal{F}_{\Gamma}(M)}Z^{2\sigma_{\varpi}} \ll [{\rm SL}_n(\mathbb{Z})\colon \Gamma]^{1+\epsilon}\label{eq:basic}
\end{equation}
when we can choose $Z\asymp Z_0(\Gamma)=[{\rm SL}_n(\mathbb{Z})\colon \Gamma]^{\frac{1}{n-1}}$. For the Borel-type congruence lattice we will need $Z_0(\Gamma_2(q)) = q^{\frac{n}{2}}$. Note that for the standard congruence subgroup $\Gamma_0(q)$ considered in \cite{Bl} one needs $Z_0(\Gamma_0(q)) =q$ while the principal congruence subgroup $\Gamma(q)$ considered in \cite{AB} requires $Z_0(\Gamma(q))=q^{n+1}$. 

\begin{remark}
These are the thresholds needed to establish Sarnak's density hypothesis. It is of course possible to go beyond and allow for larger $Z$. This is achieved in \cite{Bl} for $\Gamma_0(q)$ and in \cite{BlM, ABN} for $\Gamma(q)$. Recent work of J. Linn \cite{Lin}, extending \cite{BlM, Mi}, provides non-trivial bounds for $\textrm{GL}_n$-Kloosterman sums. It would be interesting to see if these bounds can be used to produce a subconvex density theorem for $\Gamma_2(q)$. 
\end{remark}

As in the prequels \cite{Bl, AB} we approach this estimate via the (pre)-Kuznetsov formula. We use an appropriately chosen test function, whose finite part is the indicator function $\mathbbm{1}_{I(q)}$ on the open compact subgroup $I(q)\subseteq \textrm{GL}_n(\widehat{\mathbb{Z}})$ of Borel-type. See \eqref{eq:def_Iq} and below for the precise definition of $I(q)$. We obtain an estimate of the form
\begin{equation}
\sum_{\substack{\pi \mid X_{q},\\ \Vert \mu_{\pi}\Vert \leq M}} \frac{J_{\pi}(\mathbbm{1}_{I(q)})}{l(\pi)}\cdot Z^{2\sigma_{\pi}} \ll_{M,n} [{\rm SL}_n(\mathbb{Z})\colon \Gamma]\cdot  \sum_{w\in W}\sum_{\substack{c=(c_1,\ldots,c_{n-1})\in \mathbb{Z},\\ c_i\ll Z}} \frac{\mathcal{O}_{\rm fin}^{\boldsymbol{\psi}}(c^{\ast}w)}{\vert c_1\cdots c_{n-1}\vert} .\nonumber
\end{equation}
On the spectral side we sum over cuspidal automorphic representations contributing to the spectral decomposition of $L^2(X_q)$, where $X_q$ is defined in \eqref{eq:def_Xq} below. Here we encounter the Bessel distribution $J_{\pi}(\cdot)$, see \cite[(2.8)]{Ja} or \eqref{eq:bessel_def} for a definition. The global normalizing constant $l(\pi)$ is related to $L(1,\pi,{\rm Ad})$ via the Rankin-Selberg method and can be handled by \cite{Li} because we only require an upper bound. On the geometric side we encounter the orbital integral $\mathcal{O}_{\rm fin}^{\boldsymbol{\psi}}(c^{\ast}w)$ associated to the lattice $\Gamma_2(q)$, a character $\boldsymbol{\psi}$ of $U$ and a certain diagonal matrix $c^{\ast}$ obtained from the tuple $c$. See \eqref{orb} below for a precise definition.\footnote{The orbital integrals $\mathcal{O}_{\rm fin}^{\boldsymbol{\psi}}(c^{\ast}w)$ can be expressed in terms of classical Kloosterman sums $S_{\Gamma,w}(\ast,\ast, c)$ as they appear in \cite{AB, Bl}. This translation is carried out for example in \cite{Ste}.}

The crucial input on the spectral side is a suitable lower bound for $J_{\pi}(\mathbbm{1}_{I(q)})$. We expect the bound\footnote{The heuristic behind this expectation is that the first Fourier coefficient (on average over $\pi^{I(q)}$) should be one as long as the Whittaker-period is taken with respect to the measure which is self dual with respect to the lattice $U(\mathbb{Z})\cap \Gamma$.}
\begin{equation}
	J_{\pi}(\mathbbm{1}_{I(q)}) \gg_n \dim(\pi^{{\rm SO}_n(\mathbb{R})\cdot I(q)}),\label{eq:expectation_spectral}
\end{equation}
where $\pi^{{\rm SO}_n(\mathbb{R})\cdot I(q)}$ stands for subspace of ${\rm SO}_n(\mathbb{R})\cdot I(q)$-invariant vectors in $\pi$. For the standard congruence subgroup $\Gamma_0(q)$ the analogous version of the estimate \eqref{eq:expectation_spectral} follows from newform theory, see \cite{Bl}. On the other hand, for the principal congruence subgroup $\Gamma(q)$, this estimate is established in \cite{AB, ABN} via a careful analysis of all relevant local representations. Here we establish \eqref{eq:expectation_spectral} in Lemma~\ref{low_bound_bessel} below for $\Gamma_2(q)$ (with $q=l$ prime) exploiting that representations with Iwahori fixed elements have a very nice structure.

The density theorem will now follow, if we can show that for all $\epsilon>0$ and every $w\in W$ we have
\begin{equation}
	\sum_{\substack{c=(c_1,\ldots,c_{n-1})\in \mathbb{Z},\\ c_i\ll Z}} \frac{\mathcal{O}_{\rm fin}^{\boldsymbol{\psi}}(c^{\ast}w)}{\vert c_1\cdots c_{n-1}\vert} \ll_{n,\epsilon} [{\rm SL}_n(\mathbb{Z})\colon \Gamma]^{\epsilon},\nonumber
\end{equation}
as long as $Z\leq Z_0(\Gamma_2(q))$. Note that this is immediate for the trivial Weyl element $w=1$. It is an interesting structural fact that for the lattices $\Gamma_0(q)$ and $\Gamma(q)$ only the special element 
\begin{equation}
w_{\ast} = \left( \begin{matrix} & & 1 \\  &1_{n-2} & \\ 1& &\end{matrix}\right) \nonumber
\end{equation}
survives in the corresponding ranges for $Z$. The reason for this is a very nice interplay between congruence conditions coming from $\Gamma_0(q)$ (resp. $\Gamma(q)$) and admissibility constraints on the modulus $c$ coming from the orbital integral (i.e. the Kloosterman sum). The contribution for $w_{\ast}$ can now be treated trivially using appropriate estimates for the Kloosterman sets
\begin{equation}
	X_{\Gamma}(c^{\ast}w) = \sharp \{(x,y)\in U(\mathbb{Z})\backslash U(\mathbb{Q})\times U_w(\mathbb{Q})/U_w(\mathbb{Z})\colon xc^{\ast}wy\in \Gamma\}. \nonumber
\end{equation}
Even for the relatively simple element $w_{\ast}$ deriving these estimates leads to difficult counting problems solved in \cite[Theorem~3]{Bl} and \cite[Lemma~3.4]{ABN} for $\Gamma_0(q)$ and $\Gamma(q)$ respectively. 

In the case of the Borel-type congruence subgroup $\Gamma_2(q)$ this approach fails because (at least for $n>4$) many Weyl elements survive the preliminary sieving process that takes only congruence conditions and admissibility constraints into account. We first recall that the orbital integrals factorize, so that we can isolate the ramified contribution. More precisely, we replace $c^{\ast}$ by $c^{\ast}\cdot r^{\ast}$ where $r=(r_1,\ldots r_{n-1})$ is such that $r_i\mid q^{\infty}$ (i.e. all prime factors of $r_i$ divide $q$) and $(c_i,q)=1$. Now we remember the admissibility condition on the moduli $c$ and $r$, factor the orbital integrals and estimate the unramified part (i.e. the contribution from the places not dividing $q$) trivially. We end up with
\begin{equation}
	\sum_{\substack{c=(c_1,\ldots,c_{n-1})\in \mathbb{Z},\\ c_i\ll Z}} \frac{\mathcal{O}_{\rm fin}^{\boldsymbol{\psi}}(c^{\ast}w)}{\vert c_1\cdots c_{n-1}\vert} \ll_{n,\epsilon} Z^{\epsilon} \sum_{\substack{r=(r_1,\ldots,r_{n-1})\in \mathbb{Z}^{n-1}, \\ r_i\mid q^{\infty}, \, r_i\ll Z}}\sum_{\substack{c=(c_1,\ldots,c_{n-1})\in \mathbb{Z}^{n-1}, \\ (c_i,q)=1, \, c_i\ll Z/r_i\\ (c_1r_1,\ldots, c_{n-1}r_{n-1})\text{ admissible}}} \frac{X_{\Gamma}(r^{\ast}w)}{\vert r_1\cdots r_{n-1}\vert}.
\end{equation}
For illustrative purposes we have replaced the ramified part of the orbital integral $\prod_{p\mid q}\vert \mathcal{O}^{\boldsymbol{\psi}_p}_p(c^{\ast}r^{\ast}w)\vert$, which turns out to be independent of $c$, with the cardinality of the Kloosterman set $X_{\Gamma}(r^{\ast}w)$. In the main text below we will work directly with the integral. To take admissibility into account we note that, if $w=w_{d_1,\ldots,d_k}$ is as in \eqref{eq:shape_W}, then an admissible modulus $c=(c_1,\ldots,c_{n-1})$ is essentially determined by at most $k-1$ of its entries. We roughly arrive at 
\begin{equation}
	\sum_{\substack{c=(c_1,\ldots,c_{n-1})\in \mathbb{Z},\\ c_i\ll Z}} \frac{\mathcal{O}_{\rm fin}^{\boldsymbol{\psi}}(c^{\ast}w)}{\vert c_1\cdots c_{n-1}\vert} \ll_{n,\epsilon} Z^{k-1+\epsilon}\sum_{\substack{r=(r_1,\ldots,r_{n-1})\in \mathbb{Z}^{n-1}, \\ r_i\mid q^{\infty}, \, r_i\ll Z,\\ \text{admissible}}} \frac{X_{\Gamma}(r^{\ast}w)}{\vert r_1^{1+\kappa_w(1)}\cdots r_{n-1}^{1+\kappa_w(n-1)}\vert}.
\end{equation}
where $\kappa_w(\ast)\in \{0,1\}$ depends on the shape of $w$, which in turn determines the admissibility condition. Note that this is completely independent of $\Gamma$.
 
Instead of analyzing the size of the Kloosterman sets $X_{\Gamma}(r^{\ast}w)$ individually we are now going to exploit the additional average. This can be done for example by relating the $r$-sum to the corresponding Kloosterman-set zeta function
\begin{equation}
	Z_{\Gamma}(\chi) = \sum_{\substack{r=(r_1,\ldots,r_{n-1})\in \mathbb{Z}^{n-1}, \\ r_i\mid q^{\infty}}}\frac{X_{\Gamma}(r^{\ast}w)}{\vert r_1^{1+\kappa_w(1)}\cdots r_{n-1}^{1+\kappa_w(n-1)}\vert}\chi(r^{\ast}), \label{more_class}
\end{equation}
where $\chi$ is an unramified (quasi)-character. We end up with 
\begin{equation}
	\sum_{\substack{c=(c_1,\ldots,c_{n-1})\in \mathbb{Z},\\ c_i\ll Z}} \frac{\mathcal{O}_{\rm fin}^{\boldsymbol{\psi}}(c^{\ast}w)}{\vert c_1\cdots c_{n-1}\vert} \ll Z^{r-1+\epsilon}(1+\sup_{\chi\text{ unitary}} \vert Z_{\Gamma}(\chi)\vert).\nonumber
\end{equation}
The remaining task is to estimate the Kloosterman-set zeta function $Z_{\Gamma}(\chi)$ appropriately. The corresponding estimate 
\begin{equation}
	\sup_{\chi \text{ unitary}}\vert Z_{\Gamma}(\chi)\vert \ll [{\rm SL}_n(\mathbb{Z})\colon \Gamma]^{\epsilon-\frac{r-1}{n-1}} \label{eq:des_geo_bound}
\end{equation}
is established in Proposition~\ref{prop:crucial} below an forms the technical heart of this paper.

The restriction to admissible moduli makes it harder to estimate the Kloosterman-set zeta function and one expects slightly better bounds without this restriction. See also Remark~\ref{full_rem} below. However, without restricting to admissible moduli we can not expect to win. We believe that in general the density hypothesis should follow from a trivial estimate for the orbital integrals (i.e. Kloosterman sums) together with a careful analysis of the structure of admissible moduli. Here trivial is to be understood in the sense that we ignore any cancellation that may come from the additive $\boldsymbol{\psi}$ featured in the orbital integrals of Kloosterman type after admissibility is accounted for.

\begin{remark}
Note that the density theorem will essentially follow from the estimates \eqref{eq:expectation_spectral} and \eqref{eq:des_geo_bound}. Both of these are purely local and we expect that they can be established for a large class of lattices. A particularly tractable class of lattices are the parahori-like congruence subgroups $\Gamma_P(q)$, which are the pre-image of $P(\mathbb{Z}/q\mathbb{Z})$ under the reduction map ${\rm SL}_n(\mathbb{Z})\to {\rm SL}_n(\mathbb{Z}/q\mathbb{Z})$, where $P$ is a standard parabolic subgroup.
\end{remark}

\section{Notation}

We now introduce our notation, some of which we already used above. Our conventions are mostly standard. We closely follow \cite{AB, Bl}. Besides this we make use of standard notation from analytic number theory, such as $e(x) =e^{2\pi ix}$, $f\ll g$ and so on. Boldface letters (for example $\mathbf{s}$) or symbols (for example $\boldsymbol{\epsilon}$) will usually denote tuples. Their dimension should be clear from the context.

\subsection{Groups, Roots and Matrix Decompositions}

Throughout we fix $n\in \mathbb{N}$, assume $n\geq 3$ and set $G={\rm GL}_n$. Let $U\subseteq G$ be the subgroup of unipotent upper triangular matrices and let $T\subseteq G$ be the diagonal torus. The standard Borel subgroup is $B=UT$ and the center of $G$ is denoted by $Z$. Further, set $G_0  = {\rm SL}_n$ and similarly $T_0=T\cap G_0$ as well as $B_0=B\cap G_0 = UT_0$. The identity matrix will be denoted by $\mathbf{1}_n$. When the subscript is omitted it stands for a tuple of ones: $\mathbf{1}=(1,\ldots,1)$. This similarity will hopefully not cause any confusion.

Over the real numbers we also write $Z_+\cong \mathbb{R}_{>0}$ for the subgroup of diagonal scalar matrices with positive entries. Further, we can write $T(\mathbb{R}) = V(\mathbb{R})\cdot T(\mathbb{R}_{>0})$, where $V(\mathbb{R})$ is the group of  diagonal matrices with entries in $\pm 1$. We embed $\mathbb{R}_{>0}^{n-1}$ in $T(\mathbb{R}_{>0})$ via 
\begin{equation}
	\iota(y) = {\rm diag}((y_1\cdots y_{n-j})_{1\leq j\leq n}).\label{eq:embed}
\end{equation}
The image of this embedding is denoted by $\tilde{T}(\mathbb{R}_{>0})$. (Recall that the classical upper half space is usually parametrized by $\mathcal{H}=U(\mathbb{R})\tilde{T}(\mathbb{R}_{>0})$.) The inverse $${\rm y}\colon \tilde{T}(\mathbb{R}_{>0})\to \mathbb{R}_{>0}^{n-1}$$ of $\iota$ will be useful later on. With $K_{\infty} = {\rm SO}_n(\mathbb{R})$ we can write the Iwasawa decomposition as
\begin{equation}
	G(\mathbb{R}) = U(\mathbb{R})T(\mathbb{R})K_{\infty} = U(\mathbb{R})T(\mathbb{R}_{>0}){\rm O}_n(\mathbb{R}).\nonumber
\end{equation}

Over the $p$-adic numbers the maximal compact is given by $K_p = {\rm GL}_n(\mathbb{Z}_p)$. The Iwasawa decomposition again reads $G(\mathbb{Q}_p) = U(\mathbb{Q}_p)T(\mathbb{Q}_p)K_p$. We define the local \textit{Borel-type congruence subgroup} $I_p(q)$ by
\begin{equation}
	I_p(q) = \{k\in K_p\colon [k \text{ mod }p^{v_p(q)}]\in B(\mathbb{Z}_p/q\mathbb{Z}_p) \}.\label{eq:def_Iq}
\end{equation}
Thus, if $k\in I_p(q)$, then all entries of $k$ below the diagonal are divisible by $p^{v_p(q)}$. Note that if $q=p$, then $I_p(p)$ is an Iwahori subgroup. 

Globally, we will consider the adele group $G(\mathbb{A})$ with its subgroup $G(\mathbb{A})^1$ consisting of those $g\in G(\mathbb{A})$ with $\vert \det(g)\vert_{\mathbb{A}} = 1$. We write $K_{\rm fin} = \prod_p K_p$ and $K=K_{\infty}\times K_{\rm fin}$. Further, we have the global Borel-type congruence subgroup $I(q) = \prod_p I_p(q)\subseteq K_{\rm fin}$. By strong approximation we have
\begin{equation}
	G(\mathbb{A}) = G(\mathbb{Q})\cdot (Z_+G_0(\mathbb{R})\times I(q)).\nonumber
\end{equation}
Note that 
\begin{equation}
	\Gamma_2(q) = [G_0(\mathbb{R})\times I(q)] \cap G(\mathbb{Q}), \label{eq:intersection}
\end{equation}
where $\Gamma_2(q)$ is the lattice defined in \eqref{lattice} above.

We write $\Delta$ for the roots of $G_0$. Further, let $\Delta_+$ be the subset of positive roots corresponding to our choice of $B_0$. In particular, for $\alpha\in \Delta$, we have one parameter subgroups $x_{\alpha}$ such that $U$ is generated by $x_{\alpha}$ with $\alpha \in \Delta_+$. Let $\Sigma$ be the subset of simple roots. In our case this can be made very explicit. Indeed, the roots are given by $$\Delta = \{ \alpha_{i,j}=e_i-e_j\colon i\neq j\},$$ where $$\alpha_{i,j}({\rm diag}(t_1,\ldots,t_n)) = \frac{t_i}{t_j}.$$ The one parameter group of $\alpha_{i,j}$ (with $i\neq j$) is given by
\begin{equation}
	x_{\alpha_{i,j}}(t) = \mathbf{1}_n + (t\cdot \delta_i(l)\delta_j(k))_{1\leq l,k\leq n}.\nonumber 
\end{equation}
The positive roots are then simply $\Delta_+=\{\alpha_{i,j}\colon i<j\}$ and the simple roots are $\Sigma = \{ \alpha_i = \alpha_{i,i+1}\colon i=1,\ldots, n-1\}$. We put
\begin{equation}
	h_{\alpha_i}(p) = {\rm diag}(1,\ldots ,1,p,p^{-1},1,\ldots,1), \nonumber
\end{equation}
where the $p$ appears at the $i$th position. This can be extended (multiplicatively) to all weights $\alpha = \sum_{i=1}^{n-1}k_i\alpha_i\in \bigoplus_{i=1}^{n-1}\mathbb{Z}\alpha_i$ by
\begin{equation}
	h_{\alpha} (p) = \prod_{i=1}^{n-1} h_{\alpha_i}(p^{k_i}).\nonumber
\end{equation}
Note that the image is precisely $T_0(\mathbb{Q}_p)/T_0(\mathbb{Z}_p)$.

The Weyl group will be denoted by $W$. We identify it with the subgroup of permutation matrices. To each simple root $\alpha_i$ ($i=1,\ldots, n-1$) we associate a simple reflection $s_i\in S_n$ and we abuse notation and write $s_i\in W$ for the corresponding permutation matrix. Explicitly, $s_i$ is nothing but the transposition $(i,i+1)$. Obviously, $W$ is generated by the simple reflections so that every element $w\in W$ can be written as reduced word of the form:
\begin{equation}
	w = s_{i_1}\cdots s_{i_r} \label{eq:red_word}
\end{equation}
The length of $w$ is then defined by $l(w)=r$ with $r$ as in \eqref{eq:red_word}. The longest Weyl element will be denoted by $w_l$.

Further, we define 
\begin{equation}
	U_w = w^{-1}U^{\top}w\cap U \text{ and }\overline{U}_w = w^{-1}Uw\cap U.\nonumber\nonumber
\end{equation}
In particular, we have $U=U_w\cdot \overline{U}_w = \overline{U}_w \cdot U_w$. The Bruhat decomposition reads
\begin{equation}
	G(F) = \bigsqcup_{w\in W} B(F)wU_w(F), \nonumber
\end{equation}
where $F$ is a suitable field. Note that the Weyl group acts on the roots and we set 
\begin{equation}
	R(w) = \{\alpha\in \Delta_+\colon w\alpha\in \Delta_-= \Delta\setminus \Delta_+\}.\nonumber
\end{equation}

Finally, we call the elements 
\begin{equation}
	w_{d_1,\ldots,d_k} = \left(\begin{matrix}  & & & \mathbf{1}_{d_1} \\  &  &\mathbf{1}_{d_2} & \\  & \Ddots &&\\ \mathbf{1}_{d_k} & && \end{matrix}\right) \in W \label{eq:shape_W}
\end{equation}
with $d_1+\ldots+d_k=n$ admissible. Further, given $c\in \mathbb{R}_{>0}^{n-1}$ we set
\begin{equation}
	c^{\ast} = {\rm diag}(\frac{1}{c_{n-1}},\frac{c_{n-1}}{c_{n-2}},\ldots, \frac{c_2}{c_1},c_1)\in T_0(\mathbb{R}_{>0}).\nonumber
\end{equation}
The tuple $(c,w)$ is called admissible if $w=w_{d_1,\ldots,d_k}$ is admissible and $c$ satisfies 
\begin{equation}
	c_{n-i+1}c_{n-i-1} = \pm c_{n-i}^2, \label{eg:admiss}
\end{equation}
for $i\not\in \{d_1,d_1+d_2,\ldots,d_1+\ldots+d_{k-1}\}$ where we set $c_0=c_n=1$. This condition is as in \cite[(4.3)]{AB} and will be relevant for the analysis of orbital integrals (of Kloosterman type).

\subsection{Characters}

Let $v\in \{\infty,2,3,\ldots \}$ be a place of $\mathbb{Q}$. On the torus $T(\mathbb{Q}_v)$ we define the Harish-Chandra function $H_v$ by
\begin{equation}
	H_v(t) = (\log(\vert t_1\vert_v),\ldots,\log(\vert t_n\vert_v))\in \mathbb{R}^{n}.\nonumber
\end{equation}
This is extended to $G(\mathbb{Q}_v)$ via the Iwasawa decomposition by setting $H_v(utk) = H_v(t)$. For $\lambda\in \mathbb{C}^n$ we define $$\Vert g_v\Vert_v^{\lambda} = \exp(\lambda\cdot H_v(g_v))\in \mathbb{C}^{\times}.$$ Globally, we put $\Vert g\Vert_{\mathbb{A}}^{\lambda} = \prod_v \Vert g_v\Vert_v^{\lambda}.$

Next we will introduce unramified characters of the torus $T$ at a finite place $v=p$. Essentially these are unramified characters of $(\mathbb{Q}_p^{\times})^n$ lifted to $T(\mathbb{Q}_p)$ via the isomorphism $T(\mathbb{Q}_p)\cong (\mathbb{Q}_p^{\times})^n$. However, it will be useful to think of them a bit differently. Here we closely follow \cite{Re1}. Let $\widehat{T} = \mathbb{C}^{\times}\otimes X^{\ast}(T)$ be the complex torus dual to $T$. Then we have a pairing $$\langle \, ,\, \rangle \colon T(\mathbb{Q}_p)/T(\mathbb{Z}_p)\times \widehat{T}\to \mathbb{C}^{\times}$$ given by
\begin{equation}
	\langle t,z\otimes \lambda\rangle = z^{v_p(\lambda(t))}.\nonumber
\end{equation}
This identifies $\widehat{T}$ with the group of unramified (quasi)-characters of $T$. The same pairing also identifies $T(\mathbb{Q}_p)/T(\mathbb{Z}_p)$ and the group of rational characters $X^{\ast}(\widehat{T})$ of $\widehat{T}$. Given an unramified character $\chi$ and a positive root $\alpha\in \Delta_+$ (or more generally a weight) we set
\begin{equation}
	\alpha(\chi) = \langle h_{\alpha}(p),\chi\rangle = \chi(h_{\alpha}(p)).\nonumber
\end{equation}
We observe that the following Weyl group actions are compatible:
\begin{equation}
	[w.\chi](t) = \chi(wtw^{-1}) \text{ and } [w.\alpha](\chi) = \alpha(w.\chi).\nonumber
\end{equation}
Later on we will also need the following explicit parametrization of unramified characters. Given $\mathbf{s}\in\mathbb{C}^{n-1}$ we set $$\chi_{\mathbf{s}}(c^{\ast}) = \vert c_1\vert^{s_1}\cdots \vert c_{n-1}\vert^{s_{n-1}}.$$ Note that we have
\begin{equation}
	\alpha_i(\chi_{\mathbf{s}}) = p^{s_{n-i}} \nonumber
\end{equation}
for simple roots $\alpha_i\in\Sigma$.

Finally, we need to introduce the standard (non-degenerate) character on the unipotent group $U$. To do so we first fix an additive $\mathbb{Q}$-invariant character $\psi$ of $\mathbb{A}$. This can be done so that $\psi = \psi_{\infty}\cdot \prod_p\psi_p$, where $\psi_{\infty}(x)=e(x)=e^{2\pi i x}$ is the standard character. We also assume that the $p$-adic characters $\psi_p$ are unramified (i.e. $\psi_p\vert_{\mathbb{Z}_p}\equiv 1 \not\equiv \psi_p\vert_{p^{-1}\mathbb{Z}_p}$). We can lift $\psi_{\mathbb{A}}$ to a $U(\mathbb{Q})$-invariant character $\boldsymbol{\psi}$ of $U(\mathbb{A})$ by setting
\begin{equation}
	\boldsymbol{\psi}(u) = \psi(u_{1,2}+\ldots + u_{n-1,n}).\nonumber
\end{equation}
Obviously, we have the factorization $\boldsymbol{\psi} = \boldsymbol{\psi}_{\infty}\cdot \prod_p\boldsymbol{\psi}_p$, where the local characters are defined analogously.

\subsection{Measures}

Let us start by discussing measures at the archimedean place. First we choose the probability Haar measure on ${\rm O}_n(\mathbb{R})$. The unipotent group $U(\mathbb{R})$ will be equipped with the local Tamagawa measure, which in this case is simply given by $dx=\prod_{1\leq i<j\leq n}dx_{ij}$. Finally we define a measure on $T(\mathbb{R}_{>0})$ as follows. We set
\begin{equation}
	\eta=(\frac{1}{2}j(n-j))_{1\leq j\leq n-1}\in \mathbb{Q}^{n-1} \label{eta}
\end{equation} 
and define the measure 
\begin{equation}
	d^{\ast}y = y^{-2\eta}\frac{dy_1}{y_1}\cdots \frac{d y_{n-1}}{y_{n-1}}.\nonumber
\end{equation}
This measure is pushed forward to a measure on $\tilde{T}(\mathbb{R}_{>0})$, also denoted by $d^{\ast}y$, via the embedding $\iota$ defined in \eqref{eq:embed}. The isomorphism $Z_+\cong \mathbb{R}_{>0}$ allows us to lift the measure $\frac{dz}{z}$ from $\mathbb{R}_{>0}$ to $Z_+$. The measure on $T(\mathbb{R}_{>0})$ is now defined via the decomposition $T(\mathbb{R}_{>0}) = Z_+\cdot \tilde{T}(\mathbb{R}_{>0})$. Finally, the Haar measure on $G(\mathbb{R})$ can be described using the Iwasawa decomposition by
\begin{equation}
	\int_{G(\mathbb{R})}f(g)dg = \int_{U(\mathbb{R})}\int_{T(\mathbb{R}_{>0})}\int_{{\rm O}_n(\mathbb{R})}f(xtk)dkdtdx.\nonumber
\end{equation}
Note that the usual measure $dxd^{\ast}y$ on $\mathcal{H}$ is recovered using the identification $\mathcal{H}=G(\mathbb{R})/Z_+{\rm O}_n(\mathbb{R})$. 

At a finite places $v=p$ we equip $K_p$ with the probability Haar measure and $U(\mathbb{Q}_p)$ with the Tamagawa measure.  Note that this implies that $U(\mathbb{Z}_p)$ has volume $1$. The torus $T(\mathbb{Q}_p)$ is equipped with the Haar measure $dt$ normalized by ${\rm Vol}(T(\mathbb{Z}_p),dt)=1$. The Haar measure on $G(\mathbb{Q}_p)$ can be given in Iwasawa coordinates by
\begin{equation}
	\int_{G(\mathbb{Q}_p)}f(g)dg = \int_{U(\mathbb{Q}_p)}\int_{T(\mathbb{Q}_{p})}\int_{K_p}f(xtk)\delta_B(t)dkdtdx.\nonumber
\end{equation}

On the (restricted) products $K$, $K_{\rm fin}$, $U(\mathbb{A})$ and $T(\mathbb{A})$ we put the product measures corresponding to the local measures defined above. Note that for the global group $G(\mathbb{A})$ we use the Tamagawa measure denoted by $dg$. As discussed in \cite[p. 26]{LM} it turns out that $${\rm Vol}(G(\mathbb{Q})\backslash G(\mathbb{A})^1,dg)=1.$$ Note that if we denote the product measure on $G(\mathbb{A})$ obtained from the local measures by $d_{\rm pr}g$, then we must have $C_n\cdot dg= d_{\rm pr}g.$ This constant is easily computed as follows. We use the identification $G(\mathbb{A})^1 = G(\mathbb{R})/Z_+\times {\prod_p}'G(\mathbb{Q}_p)$ and strong approximation to compute 
\begin{equation}
	C_n = {\rm  Vol}(G(\mathbb{Q})\backslash G(\mathbb{A})^1,d_{\rm pr}g)= {\rm Vol}(G_0(\mathbb{Z})\backslash \mathcal{H},dxd^{\ast}y) = \frac{1}{n}\prod_{l=2}^n\frac{\Gamma(\frac{l}{2})\zeta(l)}{\pi^{\frac{l}{2}}}.\label{def_Cn}
\end{equation}

Finally, observe that $[K_{\rm fin}\colon I(q)] = [G_0(\mathbb{Z})\colon \Gamma_2(q)] \asymp q^{\frac{n(n-1)}{2}}.$ We define $$\mathcal{V}_q = q^{\frac{n(n-1)}{2}}.$$

\subsection{Automorphic Forms and Representations}

We define the quotient 
\begin{equation}
	X_q = G(\mathbb{Q})\backslash G(\mathbb{A})^1/K_{\infty}I(q) \cong \Gamma_2(q)\backslash \mathcal{H}. \label{eq:def_Xq}
\end{equation}
and consider the corresponding $L^2$-space $L^2(X_q)$ equipped with the inner product
\begin{equation}
	\langle F,G\rangle = \int_{G(\mathbb{Q})\backslash G(\mathbb{A})^1}F(g)\overline{G(g)} dg.\nonumber
\end{equation}
The spectrum of $L^2(X_q)$ decomposes into cuspidal, residual and continuous part. In this note we focus mostly on the cuspidal bit denoted by $L^2_{\rm cusp}(X_q)$. We have the spectral decomposition
\begin{equation}
	L^2_{\rm cusp}(X_q) = \widehat{\bigoplus}_{\pi}V_{\pi},\nonumber
\end{equation}
where the sum runs over cuspidal automorphic representations and $V_{\pi}$ denotes the (finite dimensional) subspace of $K_{\infty}I(q)$-fixed elements. We write $\pi\mid X_q$ if $\pi$ is cuspidal and $V_{\pi}\neq \{0\}$. Furthermore, let $\mathcal{B}_{q}(\pi)$ be an orthogonal basis of $V_{\pi}$.

Due to Flath's theorem each cuspidal automorphic representation $\pi$ has a factorization $$\pi \cong \pi_{\infty}\otimes {\bigotimes_p}'\pi_p.$$ We write $\pi_p^{I_p(q)}$ for the subspace of $I_p(q)$ invariant elements in $\pi_p$. Note that if $p\nmid q$ this subspace is spanned by the unique (up to scaling) spherical element. We let $\mathcal{B}_q(\pi_p)$ be an orthogonal basis of $\pi_p^{I_p(q)}$. Without loss of generality we assume that $\mathcal{B}_q(\pi)$ consists of factorisable elements.

We write $\mu_{\infty}(\pi)$ for the local archimedean Langlands parameter (i.e. the Langlands parameter of $\pi_{\infty}$). For our simplistic purposes it suffices to view $\mu_{\infty}(\pi)=(\mu_1,\ldots,\mu_n)\in \mathbb{C}^n$ as an $n$-tuple of complex numbers such that $$\mu_1+\ldots+\mu_n = 0 \text{ and }\{\mu_1,\ldots,\mu_n\} = \{-\overline{\mu_1},\ldots,-\overline{\mu_n}\}.$$ Note that $\pi$ is tempered at infinity precisely when $\mu_{\infty}(\pi)\in (i\mathbb{R})^n$. It should be noted that classically the tuple $\mu_{\infty}(\pi)$ is also called the spectral parameter as it is closely connected to the eigenvalues of invariant differential operators. For comparison we note that the constant function, which is of course not a cusp form, has spectral parameter $(\frac{n-1}{2},\frac{n-3}{2},\ldots,\frac{3-n}{2},\frac{1-n}{2})$. We define
\begin{equation}
	\sigma_{\infty}(\pi) = \max_{i=1,\ldots,n} \vert \Re(\mu_i)\vert.\nonumber
\end{equation}
We are now ready to define the counting function
\begin{equation}
	N_{\infty}(\sigma;q,M) = \sharp\{ \pi\mid X_q\colon \Vert\mu_{\infty}(\pi)\Vert\leq M \text{ and }\sigma_{\infty}(\pi)\geq \sigma  \}.\nonumber
\end{equation}
This relates to the classical counting function $N_{\Gamma_2(q)}(\sigma;M)$ defined in the introduction as follows. First, a standard adelisation procedure shows that
\begin{equation}
	N_{\Gamma_2(q)}(\sigma;M) = \sum_{\substack{\pi \mid X_q,\\ \Vert \mu_{\infty}(\pi)\Vert \leq M,\\ \sigma_{\infty}(\pi)\geq \sigma}} \dim_{\mathbb{C}}(V_{\pi}).\nonumber
\end{equation}
Second, a local computation, see Lemma~\ref{5.1} below, shows that $\dim_{\mathbb{C}}(V_{\pi})\ll_n 1$. Thus we have
\begin{equation}
	N_{\Gamma_2(q)}(\sigma;M) \asymp_n N_{\infty}(\sigma;q,M).\nonumber
\end{equation}
For the rest of this note we will only be concerned with the counting function $N_{\infty}(\sigma;q,M)$.

\section{The Iwahori-Hecke Algebra and Unramified Principal Series}

The goal of this section to recall some results from the theory of Iwahori-Hecke algebras that will be needed later on. Doing so we will mostly follow the exposition from \cite{Re1, Re2}. We will work exclusively over $\mathbb{Q}_l$ and $\mathbb{Z}_l$ for a fixed prime $l$. This suffices for the goals of this note, but the results are of course true for general non-archimedean local fields $F$ with valuation ring $\mathcal{O}$. 

Note that $I_l(l)$ is the Iwahori subgroup determined by our choice of positive roots. To shorten notation we will set $I_0=I_l(l)\subseteq K_l$ throughout this section.

For two functions $S,T\in \mathcal{C}_c^{\infty}(G(\mathbb{Q}_l))$ the convolution is defined as usual by
\begin{equation}
	[T\ast S](g)=\int_{G(\mathbb{Q}_l)}T(gh^{-1})S(h)dh.\nonumber
\end{equation}
We also set $S^{\vee}(g) = \overline{S(g^{-1})}$. 

\begin{remark}
Suppose $\pi$ is a cuspidal automorphic representation, $\phi\in \pi$ and $T\in \mathcal{C}_c^{\infty}(G(\mathbb{Q}_l))$, then we write $$T\phi = \int_{G(\mathbb{Q}_l)}T(g_l)\pi(g_l)\phi dg_l,$$ where $G(\mathbb{Q}_l) \subseteq G(\mathbb{A})$ in the obvious way.
\end{remark}

The Iwahori-Hecke algebra is defined as
\begin{equation}
	\mathcal{H}_I = \mathcal{C}_c^{\infty}(I_0\backslash G(\mathbb{Q}_l)/I_0)
\end{equation}
Due to the Iwahori decomposition
\begin{equation}
	G(\mathbb{Q}_l)=\bigsqcup_{w\in W} B(\mathbb{Q}_l) w I_0\nonumber
\end{equation}
a linear basis of $\mathcal{H}_I$ is given by
\begin{equation}
	T_{aw}=[K_l\colon I_0]\cdot \mathbbm{1}_{I_0awI_0} \nonumber
\end{equation}
for $a\in T(\mathbb{Q}_l)/T(\mathbb{Z}_l)$ and $w\in W$.\footnote{The multiplication with $[K_l\colon I_0]$ is included to be compatible with our main references \cite{Re1, Re2}. Indeed, there the Haar measure is normalized to give volume $1$ to $I_0$, but in accordance to the rest of this manuscript our Haar measure satisfies ${\rm Vol}(K_l) = 1$. To make up for this we chose to include the corresponding normalizing factor here.}

The subalgebra consisting of functions supported on $K_l$ will be denoted by $\mathcal{H}_W$. This algebra is finite dimensional and a convenient linear basis is given by
\begin{equation}
	\{ T_w\colon w\in W\}.\nonumber
\end{equation}

\begin{remark}
The full Iwahori-Hecke algebra can be written as a tensor product $\mathcal{H}_I=\mathcal{H}_W\otimes \Theta$ where $\Theta$ is commutative (and isomorphic to the group algebra $\mathbb{C}[T(\mathbb{Q}_l)/T(\mathbb{Z}_l)]$). Since for our purposes the algebra $\Theta$ is mostly irrelevant we will not discuss it any further here. For more details we refer to \cite{Re1} and the references within.
\end{remark}

Let $\chi$ be an unramified character of $T(\mathbb{Q}_l)$ then we define the induced representation
\begin{equation}
	I(\chi) = {\rm Ind}_{B(\mathbb{Q}_l)}^{G(\mathbb{Q}_l)}(\chi) \nonumber
\end{equation}
consisting of smooth (i.e. locally constant) functions $f\colon G(\mathbb{Q}_l)\to \mathbb{C}$ such that $f(bg) = [\chi\cdot \delta^{\frac{1}{2}}](b)f(g)$ for all $b\in B(\mathbb{Q}_l)$ and all $g\in G(\mathbb{Q}_l)$. We write $I(\chi)^{I_0}$ for the subspace of Iwahori-fixed elements. The functions
\begin{equation}
	\phi_w^{\chi}(g) = \begin{cases}
			[\chi\cdot\delta^{\frac{1}{2}}](a) &\text{ if }g=auwk \text{ for }a\in T(\mathbb{Q}_l),\, u\in U(\mathbb{Q}_l) \text{ and }k\in I_0,\\
			0 &\text{ else.}
	\end{cases}
\end{equation}
form the so called standard basis of $I(\chi)^{I_0}$.

Given $f\in \mathcal{H}_I$ the corresponding convolution operator 
\begin{equation}
	R(f)\colon I(\chi) \to I(\chi),\,  h\mapsto \left[x  \mapsto \int_{G(\mathbb{Q}_l)}h(xg)f(g)dg\right]
\end{equation} 
preserves the subspace $I(\chi)^{I_0}$. Thus $I(\chi)^{I_0}$ is a $\mathcal{H}_I$-module. We will abuse notation and write $T_w\phi = R(T_w)\phi$ for $\phi\in I(\chi)^{I_0}$ and $T_w\in \mathcal{H}_W$ as defined above. Note that the action of $\mathcal{H}_W$ on $I(\chi)^{I_0}$ is independent of $\chi$. The action can be described quite explicitly as follows:
\begin{itemize}
	\item For a simple reflection $s\in W$ we have
	\begin{equation}
		T_s\phi_w^{\chi} = \begin{cases}
			\phi_{ws}^{\chi} &\text{ if }ws>w,\\
			l\phi_{ws}^{\chi}+(l-1)\phi_w^{\chi} &\text{ if }ws<w.
		\end{cases} \nonumber
	\end{equation}
	\item We have $T_{w_1w_2} = T_{w_1}T_{w_2}$ if $l(w_1w_2) = l(w_1)+l(w_2)$ and $T_1$ is the identity operator.
\end{itemize}

We conclude this section by introducing the intertwining operators $\mathcal{A}_w^{\chi}\colon I(\chi) \to I(w.\chi)$. These operators are defined by the integral
\begin{equation}	
	[\mathcal{A}_w^{\chi}(f)](x) = \int_{U_w(\mathbb{Q}_l)}f(wux)du
\end{equation}
when $\chi$ is sufficiently dominant and understood by analytic continuation otherwise. Their properties are similar to those of the operators $T_w$. Indeed, by \cite[(3.4)]{Ca} or \cite[pp. 321-322]{Re1} we have
\begin{itemize}
	\item $\mathcal{A}_{w_1w_2}^{\chi} = \mathcal{A}_{w_2}^{w_1.\chi}\mathcal{A}_{w_1}^{\chi}$ for $l(w_1w_2)=l(w_1)+l(w_2)$ and
	\item For a simple root reflection $s\in W$ associated to $\alpha\in \Sigma$ (i.e. $l(s)=1$) we have 
	\begin{equation}
		\mathcal{A}_s^{\chi}\phi_w^{\chi} = \begin{cases}
			\alpha(\chi)\cdot c_{\alpha}(\chi)\phi_{w}^{s.\chi}+l^{-1}\phi_{sw}^{s.\chi} &\text{ if }sw>w,\\
			c_{\alpha}(\chi)\phi_w^{s.\chi}+\phi_{sw}^{s.\chi} &\text{ if }sw<w, 
		\end{cases}\label{rec_intertwiners}
	\end{equation}
	for $c_{\alpha}(\chi) = \frac{1-l^{-1}}{1-\alpha(\chi)}$.
\end{itemize}

The action of the intertwining operators on the spherical vector $$\phi_+^{\chi} = \sum_{w\in W} \phi_w^{\chi}$$ is well understood. It is given by a version of the Gindikin-Karpelevich formula due to Langlands \cite{Lan}. Indeed, one as
\begin{equation}
	\mathcal{A}_w^{\chi}\phi_+^{\chi} = C_w(\chi)\phi_+^{w.\chi}\label{eq:intertwiner_eig_sph}
\end{equation}
for
\begin{equation}
	C_w(\chi)= \prod_{\alpha\in R(w^{-1})} \frac{1-l^{-1}\alpha(\chi)}{1-\alpha(\chi)}. \nonumber
\end{equation}
See also \cite{Ca} or \cite[p. 19]{Re2}.

\begin{remark}
In general, it is quite difficult to compute the expression $\mathcal{A}_{w_1}^{\chi}\phi_{w_2}^{\chi}$ explicitly. Of course this a purely combinatorial problem, but it has many interesting ramifications, one being Casselman's basis problem. We refer to \cite{BM1, BM2, AMSS} for extensive treatment of the latter.
\end{remark}

\section{On the Bessel Distribution of Iwahori-Spherical Representations}

Bessel distributions (at the finite place $l$) will appear on the spectral side of our pre-Kuznetsov formula. After briefly introducing them we will bound them from below when evaluated at a suitable test-function $T\in \mathcal{H}_W$. 

Throughout this section we work locally at the place $l$ (i.e. over the field $\mathbb{Q}_l$). To simplify notation we will thus omit the subscript $l$. Let $\pi$ be an irreducible smooth admissible unitary representation of $\textrm{GL}_n(\mathbb{Q}_l)$ with Iwahori-fixed vectors. Later on we will further assume that $\pi$ is generic. We start with the following lemma which characterizes representations with Iwahori-fixed vectors.

\begin{lemma}\label{5.1}
Let $\pi$ be an irreducible smooth admissible unitary representation with Iwahori fixed vectors (i.e. $\pi^{I_0} \neq\{0\}$). Then $\pi$ can be realized as the unique irreducible sub-representation of an unramified principal series $I(\chi)$. We have
\begin{equation}
	\dim_{\mathbb{C}}\pi^{I_0} \ll_n 1.\nonumber
\end{equation}
If $\pi$ is generic, then $$\pi\cong {\rm Ind}_{P(\mathbb{Q}_l)}^{G(\mathbb{Q}_l)}(\vert \cdot\vert^{ir_1}\cdot{\rm St}_{n_i}\otimes \ldots\otimes \vert \cdot\vert^{ir_s}\cdot{\rm St}_{n_s}),$$ where $P$ is the standard parabolic subgroup of $G$ associated to the partition $n=n_1+\ldots+n_s$, $r_1,\ldots,r_s\in \mathbb{R}$ and ${\rm St}_{n_i}$ denotes the Steinberg representation of ${\rm GL}_{n_i}(\mathbb{Q}_l)$.
\end{lemma}
\begin{proof}
First, by \cite[Proposition~2.6]{Ca}, there is some unramified principal series representation $I(\chi)$ and a $G(\mathbb{Q}_l)$-embedding of $\pi$ into $I(\chi)$. Since $\dim_{\mathbb{C}}I(\chi)^{I_0} = \sharp W = n!$ we must have $\dim_{\mathbb{C}} \pi_l^{I_0}\leq n!$.

Finally if $\pi$ is generic, the claim follows from the classification result \cite[Theorem~9.7]{Ze}.
\end{proof}

From now on we assume that $\pi$ is unitary and generic. We fix an intertwiner $v\mapsto W_v$ from $\pi$ to its $\boldsymbol{\psi}$-Whittaker model ${\rm Wh}(\pi)$. This amounts to a choice of Whittaker functional. Further, the inner product on ${\rm Wh}(\pi)$ is
\begin{equation}
	\langle W_1,W_2\rangle_{{\rm Wh}(\pi)} = \int_{U^{(n-1)}(\mathbb{Q}_l)\backslash G^{(n-1)}(\mathbb{Q}_l)}W_1\left(\left(\begin{matrix} h& 0\\0&1\end{matrix}\right)\right)\cdot\overline{W_2\left(\left(\begin{matrix} h &0\\0&1\end{matrix}\right)\right)}dh, \nonumber 
\end{equation} 
for $W_1,W_2\in {\rm Wh}(\pi)$. Here $G^{(n-1)}=\textrm{GL}_{n-1}$ and $U^{n-1}\subseteq G^{(n-1)}$ is the standard unipotent radical. Define  
\begin{equation}
	J(W)	= \int_{U(\mathbb{Q}_l)}^{\rm st} \langle \pi(x)W,W\rangle_{{\rm Wh}(\pi)} \boldsymbol{\psi}(x)^{-1}dx, \label{def_Il}
\end{equation}
for $W\in{\rm Wh}(\pi)$, 

We now introduce the Bessel distribution. To do so we fix an orthogonal basis $\mathcal{B}(\pi)$. We recall that $\mathcal{B}_l(\pi)$ denotes a basis of $\pi^{I_0}$ and arrange that $$\mathcal{B}_l(\pi)\subseteq \mathcal{B}(\pi).$$ The Bessel distribution is then defined by
\begin{equation}
	J_{\pi}(F) = \sum_{v\in \mathcal{B}(\pi)}\frac{W_{\pi(F)v}(1)\overline{W_v(1)}}{\langle W_v,W_v\rangle_{{\rm Wh}(\pi)}}.\label{eq:bessel_def}
\end{equation}
For more details we refer to \cite{BaM}. We can relate the summands in the definition of the Bessel distribution to the integrals $J(\cdot)$:

\begin{lemma}\label{recognice_bessel}
For $T\in \mathcal{H}_W$ we have
\begin{equation}
	\sum_{v\in \mathcal{B}_l(\pi)}\frac{J(W_{Tv})}{\langle W_v,W_v\rangle_{{\rm Wh}(\pi)}} = J_{\pi}(T^{\vee}\ast T).
\end{equation}
\end{lemma}
\begin{proof}
We first note that, since both $T$ and $T^{\vee}$ are bi-$I_0$-invariant, we can write
\begin{equation}
	J_{\pi}(T^{\vee}\ast T) = \sum_{v\in \mathcal{B}_l(\pi)}\frac{W_{\pi(T^{\vee}\ast T)v}(1)\overline{W_v(1)}}{\langle W_v,W_v\rangle_{{\rm Wh}(\pi)}}
\end{equation}
As in \cite[Lemma~2.2]{Ja} we further see that 
\begin{equation}
	J_{\pi}(T^{\vee}\ast T) = \sum_{v\in \mathcal{B}_l(\pi)}\frac{W_{\pi(T)v}(1)\overline{W_{\pi(T)v}(1)}}{\langle W_v,W_v\rangle_{{\rm Wh}(\pi)}}
\end{equation}
We conclude by applying \cite[Lemma~4.4]{LM}.
\end{proof}

This allows us to derive the following useful bound.

\begin{lemma}\label{low_bound_bessel}
Let $\pi$ be an irreducible smooth admissible unitary generic representation with $\pi^{I_0}\neq \{0\}$. Then we have $J_{\pi}(T_1) \gg 1$.
\end{lemma}
\begin{proof}
We start by considering $\pi_l = {\rm St}_n$. In this case, $\pi_l^{I_0}$ is one dimensional and it is straight forward to see that $$J_{{\rm St}_n}(T_1) = 1.$$ More precisely, one can simply use the formula for the corresponding matrix coefficient given in equation (3) in the proof of Proposition~5.3 in \cite{Bo}. 

The general case can be reduced to this one by means of parabolic induction. Following the discussion above we can assume that 
\begin{equation}
	\pi = {\rm Ind}_{P(\Bbb{Q}_l)}^{{\rm GL}_n(\Bbb{Q}_l)}(\sigma) \nonumber
\end{equation}
with $\sigma =(\tau_1,\ldots,\tau_d)$ and $\tau_d = \vert\cdot \vert^{ir_i}{\rm St}_{n_i}$. Each $\tau_i$ has a unique Iwahori fixed vector $v_{\tau_i}$. Therefore
\begin{equation}
v_{\sigma} = v_{\tau_1}\otimes \ldots \otimes v_{\tau_d} \nonumber
\end{equation}
is the unique (non-trivial) $I_0\cap M(\Bbb{Q}_l)$-fixed vector of $\sigma$. We define $J_{\sigma}$ in the obvious way, see the display above Lemma~6.9 in \cite{AB}.\footnote{Note that in \cite{AB} the notation $\mathcal{S}_{\sigma}$ is used instead of $J_{\sigma}$.} Our goal is  to reduce  $J_{\pi}$ to $J_{\sigma}$ in the case of parabolically induced representations. The reduction argument is similar to the one given in \cite[Section~6.3]{AB}. We can work with the orthonormal basis
\begin{equation}
f_w(g) = {\rm Vol}(P(\mathbb{Z}_l)wI_0,dk)^{-\frac{1}{2}}\cdot \begin{cases}
\delta_P(m)\sigma(m)v_{\sigma} & \text{ if }g=nmwk \text{ for }n\in N, \, m\in M \text{ and }k\in I_0, \\
0 &\text{ else}
\end{cases}\nonumber
\end{equation}
indexed by $w\in W/W_M$. Note that here we are using that ${\rm GL}_n(\Bbb{Q}_l) = \bigsqcup_{w\in W/W_M} P(\mathbb{Q}_l)w I_0$. 

In general, computing the corresponding Whittaker periods is combinatorial quite involved. However, since we are assuming that $\pi_l$ is unitary and irreducible, we have positivity. Thus, it suffices to establish that 
\begin{equation}
\int_{U(\Bbb{Q}_l)}^{{\rm st}} \langle  \pi(u)f_{w_P},f_{w_P}\rangle_{\pi} \boldsymbol{\psi}(u)^{-1} du \gg 1 \nonumber
\end{equation}
for $w_P$ as in the proof of \cite[Lemma~6.9]{AB}. The computation of this period is straight forward. First one computes, essentially following \cite[Lemma~6.8]{AB} that
\begin{equation}
\langle \pi(u)f_{w_P},f_{w_P}\rangle_{\pi} = \int_{K_l} \langle f_{w_P}(ku),f_{w_P}(k)\rangle_{\sigma}dk = \langle f(w_P u),v_{\sigma}\rangle_{\sigma},\nonumber
\end{equation}
for $u\in U(\Bbb{Q}_l)$. Now one considers the condition $w_Mu\in P(\Bbb{Q}_l)w_P I_0$. Decomposing $u=u_{M^{\rm op}}u_{N^{\rm op}} \in M^{\rm op}N^{\rm op}\cap U$ one obtains $u_{N^{\rm op}}\in I_0$. With this at hand we get
\begin{align}
J_{\pi}(T_1) &\geq \int_{U(\Bbb{Q}_l)}^{{\rm st}} \langle  \pi(u)f_{w_P},f_{w_P}\rangle_{\pi} \boldsymbol{\psi}(u)^{-1}du \nonumber \\
&= \int_{U(\Bbb{Q}_l)}^{{\rm st}} \langle f_{w_P}(w_Pu),v_{\sigma}\rangle_{\sigma} \boldsymbol{\psi}(u)^{-1}du \nonumber \\
&= \int_{U(\Bbb{Q}_l)\cap M^{\op}}^{{\rm st}} \langle  \sigma(w_Pu_{M^{\rm op}}w_{P}^{-1})v_{\sigma},v_{\sigma}\rangle_{\sigma} \boldsymbol{\psi}(u_{M^{\rm op}})^{-1}du_{M^{\rm op}}= J_{\sigma^{\rm op}}(T_1) = 1.\nonumber
\end{align}
This computation is similar to the final steps in the proof of \cite[Lemma~6.9]{AB}. The result is analogous to \cite[Lemma~4.1 (b)]{CSh}.

\end{proof}

\section{On P-adic Orbital Integrals of Kloosterman Type}

In this section, we define and study (local) orbital integrals of Kloosterman type. These will appear in the geometric expansion of the Fourier coefficients of Poincar\'e series. 

Since we will exclusively work locally over the field $\mathbb{Q}_p$ we drop the subscript $p$ throughout this section. Note that we will encounter the unramified situation $p\neq l$ as well as the ramified situation $p=l$.

Let $F\colon G(\mathbb{Q}_p)\to \mathbb{C}$ be a function such that
\begin{equation}
	F(g) = \begin{cases}
		\boldsymbol{\psi}(u)Q(k) &\text{ if }g=utk \text{ with }u\in U(\mathbb{Q}_p),\, t\in Z(\mathbb{Q}_p)T(\mathbb{Z}_p) \text{ and }k\in K_p,\\
		0 &\text{ else.}
	\end{cases}\nonumber
\end{equation}
for this to be well defined $Q\colon K_p\to \mathbb{C}$ must be $T(\mathbb{Z}_p)$ invariant. Then the orbital integral is defined as
\begin{equation}
	\mathcal{O}^{\boldsymbol{\psi}}(tw,F) = \int_{U_w(\mathbb{Q}_p)}F_Q(twu)\boldsymbol{\psi}(u)^{-1}du.\label{p_orb}
\end{equation}
In the global context these will be denoted $\mathcal{O}_p^{\boldsymbol{\psi}_p}(tw,F_p)$ to indicate that they are defined at the place $p$.

\begin{remark}
In practice we will encounter two specific situations. First, for $p=l$ we can take $Q\in \mathcal{H}_W$. Second, the simplest choice is the characteristic function $Q=\mathbbm{1}_{K_p}$ of $K_p$. In this case, one can relate the corresponding orbital integral $\mathcal{O}(tw,F)$ to the classical Kloosterman sum associated to the Weyl element $w$. This is worked out in \cite[Theorem~2.12]{Ste}. However, to spare the translation process we prefer to work with the orbital integral directly.
\end{remark}

For $w=1$ the orbital integral is easy to compute:
\begin{equation}
	\mathcal{O}^{\boldsymbol{\psi}}(t,F) = \delta_{t\in Z(\mathbb{Q}_p)T(\mathbb{Z}_p)}F(\mathbf{1}_n).\nonumber
\end{equation}
For general $w$ we record the trivial bound
\begin{equation}
		\vert \mathcal{O}^{\boldsymbol{\psi}}(tw,F)\vert \leq \mathcal{O}^{\mathbf{1}}(tw,\vert F\vert).\nonumber
\end{equation}
We note that $t\mapsto \mathcal{O}^{\mathbf{1}}(tw,\cdot)$ is $Z(\mathbb{Q}_p)T(\mathbb{Z}_p)$-invariant. Thus it makes sense to define
\begin{equation}
	Z_w(F;\chi) = \sum_{t\in Z(\mathbb{Q}_p)T(\mathbb{Z}_p)\backslash T(\mathbb{Q}_p)} \mathcal{O}^{\mathbf{1}}(tw,\vert F\vert)\delta_B^{-\frac{1}{2}}(t)\chi^{-1}(t), \label{def_Kloosterman_zeta}
\end{equation}
for an unramified character $\chi$ which is trivial on $Z(\mathbb{Q}_p)$. Ideas taken from \cite[Section~2]{DR} will show that this converges for $\chi$ sufficiently dominant. Define
\begin{equation}
	\Phi_{Q}^{\chi}(g) = \int_{Z(\mathbb{Q}_p)\backslash T(\mathbb{Q}_p)} \vert F(tg)\vert \delta_B^{-\frac{1}{2}}(t)\chi^{-1}(t)dt,\nonumber
\end{equation}
where the notion is justified since $\vert F\vert$ depends only on the function $Q\colon K_p\to \mathbb{C}$ used in its definition. We see directly that $\Phi_Q^{\chi}\in I(\chi)$. Thus we can write
\begin{equation}
	Z_w(F;\chi) = [\mathcal{A}_w^{\chi}\Phi_Q^{\chi}](\mathbf{1}_n).\label{eq:zeta_as_int}
\end{equation}
This reduces the analysis of $Z_w(F,\cdot)$ to the analysis of $\mathcal{A}_w^{\chi}$. In general, this is still a hard problem, but for certain $Q$ this is doable. We record the following straight forward statement:

\begin{lemma}\label{lm:another_lemm}
If $Q=\mathbbm{1}_{K_p}$, then $$\Phi_Q^{\chi} = \phi_+^{\chi}\in I(\chi)$$ is the spherical vector. If $p=l$ and $Q=T_w\in \mathcal{H}_W$, then we have
\begin{equation}
	\Phi_{Q}^{\chi} = \mathcal{V}_l\cdot \phi_w^{\chi}\in I(\chi).\nonumber
\end{equation}
\end{lemma}

For $Q=\mathbbm{1}_{K_p}$ we can thus insert \eqref{eq:intertwiner_eig_sph} in \eqref{eq:zeta_as_int} and obtain
\begin{equation}
	Z_w(F;\chi) = C_w(\chi).\label{nonarch_vers_comp}
\end{equation}
By Mellin inversion (or expanding $C_w(\cdot)$ and comparing coefficients) we obtain the following result.

\begin{lemma}\label{triv_bound_KS}
For $Q=\mathbbm{1}_{K_p}$ and $c\in \mathbb{Q}_{\neq 0}^{n-1}$ we have
\begin{equation}
	\mathcal{O}^{\mathbf{1}}(c^{\ast}w,F) \ll \vert c_1\cdots c_{n-1}\vert^{1+\epsilon}.\nonumber
\end{equation}
and $\mathcal{O}^{\mathbf{1}}(c^{\ast}w,F) = 0$ for $c\not\in \mathbb{Z}_p^{n-1}$.
\end{lemma}

This is of course nothing but the trivial bound for Kloosterman sums due to R. Dabrowski and M. Reeder in disguise. (See \cite[Theorem~0.3]{DR}  for the local formulation or \cite[(4.6)]{AB} in the global case.)

For the rest of this section we assume that $p=l$ and $Q\in \mathcal{H}_W$. We start by a simple observation concerning the support of $c\mapsto \mathcal{O}(vc^{\ast}w,F)$.

\begin{lemma}
Let $Q\in \mathcal{H}_W$ and $c\in \mathbb{Q}_{\neq 0}^{n-1}$, then $\mathcal{O}^{\boldsymbol{\psi}}(vc^{\ast}w,F) = 0$ unless $c\in \mathbb{Z}_p^{n-1}$.
\end{lemma}
\begin{proof}
Without loss of generality we can assume that $T=T_u$ for $u\in W$. The result will be obtained by analyzing $[\mathcal{A}_w^{\chi}\phi_{u}^{\chi}](\mathbf{1}_n)$. We start by expanding
\begin{equation}
	\mathcal{A}_w^{\chi}\phi_{u}^{\chi} = \sum_{v\in W}a_{v,u}(w;\chi)\cdot \phi_v^{w\chi} \nonumber
\end{equation}
for coefficients $a_{v,u}(w;\chi)\in \mathbb{C}$. After recalling the support of the functions $\phi_v^{w\chi}$ we obtain
\begin{equation}
	Z_w(T_u;\chi) = \mathcal{V}_l\cdot [\mathcal{A}_w^{\chi}\phi_{u}^{\chi}](\mathbf{1}_n) = \mathcal{V}_l\cdot a_{1,u}(w;\chi).\nonumber
\end{equation}
The coefficients $a_{v,u}(w;\chi)$ can be computed recursively using \eqref{rec_intertwiners}. We will do so in the particular case $u=1$ in Lemma~\ref{general_expr} below. However, for the moment it is sufficient to observe that there are sets $S_1,S_2\subseteq R(w^{-1})$ such that 
\begin{equation}
	a_{1,u}(w;\chi) = a_{1,u}(w) \cdot \prod_{\alpha\in S_1}\alpha(\chi) \cdot \prod_{\alpha\in S_2}c_{\alpha}(\chi),\nonumber
\end{equation}
for some new constant $a_{1,u}(w)$ independent of $\chi$. In view of the definition of $Z_w(T_u;\chi)$ we have obtained
\begin{equation}
	\sum_{t\in Z(\mathbb{Q}_p)T(\mathbb{Z}_p)\backslash T(\mathbb{Q}_p)} \mathcal{O}^{\mathbf{1}}(tw,\vert F\vert)\delta_B^{-\frac{1}{2}}(t)\chi^{-1}(t) = \mathcal{V}_l\cdot a_{1,u}(w) \cdot \prod_{\alpha\in S_1}\alpha(\chi) \cdot \prod_{\alpha\in S_2}c_{\alpha}(\chi).\nonumber
\end{equation}
The right hand side can be expanded as a power series and the claimed vanishing result follows after comparing coefficients.
\end{proof}

Let $(c,w)$ be admissible. Thus $w=w_{d_1,\ldots,d_k}$ and $c$ satisfies \eqref{eg:admiss}. We observe that modulo $T_0(\mathbb{Z}_l)$ the $(n-1)$-tuple $c\in \mathbb{Z}^{n-1}$ is completely determined by $c_i$ with $i\in \{n-d_1,n-d_1-d_2,\ldots, n-d_1-\ldots-d_{k-1}\}$.\footnote{It is easier to see that $c_{d_1}, c_{d_1+d_2},\ldots$ (essentially) determine $c$ when $(c,w)$ is admissible. But our choice of indices also works and will be crucial to the argument.} To detect this condition we will use $k-1$-tuples $\mathbf{s}\in \mathbb{C}^{k-1}$ and associate the (unramified) characters $\chi_{\mathbf{s}}^{\rm ad}$ defined by
\begin{equation}
	\chi_{\mathbf{s}}^{\rm ad}(c^{\ast}) = \prod_{i=1}^{k-1}\vert c_{n-d_1-\ldots-d_i}\vert_l^{-s_i}.\nonumber
\end{equation}
Note that this character is defined in terms of the parametrization $c \mapsto c^{\ast}$ of $T_0(\mathbb{Q}_l)$. One technical difficulty that will arise below is that the characters $\chi_{\mathbf{s}}^{\rm ad}$ are not dominant (unless $w$ is the long element). We will need the following result, which turns out to be combinatorial quite involved:

\begin{prop}\label{prop:crucial}
For $w=w_{d_1,\ldots,d_r}$ and $\mathbf{s}\in \mathbb{C}^{k-1}$ we have
\begin{equation}
	\sup_{\Re(\mathbf{s}) = \mathbf{1}+\boldsymbol{\epsilon}}\vert \mathcal{A}_{w}^{\chi^{\rm ad}_{\mathbf{s}}}\phi_1^{\chi^{\rm ad}_{\mathbf{s}}}(\mathbf{1}_n) \vert \ll l^{-\frac{n(k-1)}{2}}. \nonumber
\end{equation}
\end{prop}

\begin{remark}\label{full_rem}
Allowing general unramified characters associated to $\mathbf{s}\in \mathbb{C}^{n-1}$ we expect the estimate
\begin{equation}
	\sup_{\Re(\mathbf{s}) = \mathbf{1}+\boldsymbol{\epsilon}}\vert \mathcal{A}_{w}^{\chi_{\mathbf{s}}}\phi_1^{\chi_{\mathbf{s}}}(\mathbf{1}_n) \vert \ll l^{-l(w)} \nonumber
\end{equation}
to hold. Note that for the long Weyl element these two estimates are of course the same, because all moduli are admissible. 
\end{remark}

The proof of Proposition~\ref{prop:crucial} will take up the remainder of this section and is the technical core of this paper. Unfortunately we could not find a soft abstract argument and therefore we will get dirty hands. 

Given $w=w_{d_1,\ldots, d_k}$ we associate the standard parabolic subgroup $P_w=M_wN_w$ with Levi component
\begin{equation}
	M_w = \left(\begin{matrix} {\rm GL}_{d_1} &  & \\  & \ddots & \\  & & {\rm GL}_{d_k} \end{matrix}\right).\nonumber 
\end{equation}
We see that
\begin{equation}
	N_w = \prod_{\alpha\in R(w^{-1})}U_{\alpha},\nonumber
\end{equation}
where $U_{\alpha} =\{x_{\alpha}(t)\}$. We start with some simple observations:
\begin{itemize}
	\item It is well known that $l(w) = \sharp R(w^{-1})$;
	\item The simple roots in $R(w^{-1})$ are precisely $\alpha_i$ for $i=d_k, d_k+d_{k-1}, \ldots$. We write this as
	\begin{equation}
		R(w^{-1})\cap \Sigma = \{\alpha_i\colon i=d_1,d_1+d_2,\ldots, d_1+\ldots+d_{k-1} \}.\nonumber
	\end{equation}
	\item If $(c,w)$ is admissible, then $c^{\ast} \in Z_{M_w}(\mathbb{Q}_l)\cdot T_0(\mathbb{Z}_l)$, where $Z_{M_w}$ denotes the center of $M_w$.
\end{itemize}

Our analysis of $\mathcal{A}_w^{\chi}\phi_1^{\chi}$ crucially depends on a suitable representation of $w$ as a reduced word in simple reflections $s_i$. Such a word gives rise to a numbering of $R(w^{-1})$, which we imagine as numbering of \textit{free squares} of $N_w$. 

\begin{example}
Of course the simplest examples are given by $w_{1,n-1} = s_1\cdots s_{n-1}$ and $w_{n-1,1} = s_{n-1}\cdots s_1$. The weights are accordingly numbered by $R(w_{1,n-1}^{-1}) = \{\beta_1=\alpha_1,\beta_2=\alpha_1+\alpha_2,\ldots ,\beta_{n-1}=\alpha_1+\ldots+\alpha_{n-1}\}$ and $R(w_{n-1,1}^{-1}) = \{\beta_1=\alpha_{n-1},\beta_2=\alpha_{n-1}+\alpha_{n-2},\ldots ,\beta_{n-1}=\alpha_{n-1}+\ldots+\alpha_{1}\}$. We visualize this as:
\begin{equation}
	N_{w_{1,n-1}} = \left(\begin{matrix} 1 & \beta_1 & \dots & \beta_{n-1} \\  & 1 & & \\& &\ddots & \\ & & & 1\end{matrix}\right) \text{ and }N_{w_{n-1,1}} = \left(\begin{matrix} 1 & & & \beta_{n-1} \\ & \ddots & &\vdots \\  & & 1&\beta_1 \\ & & & 1\end{matrix}\right).\nonumber
\end{equation}
\end{example}

First we need to fix a convenient reduced expression for $w$. This can be done as follows:
\begin{enumerate}
	\item We number the rows of $N_w$ by $R_0,\ldots,R_{n-d_r-1}$ from bottom to top;
	\item Let $l_i = l(R_i)$ denote the length of the row, by which me mean the number of  free entries a matrix in $N_w$ can have in this row;
	\item To each row we attach a reduced word $$w_i=w_{R_i} = s_{n-d_r-i}s_{n-d_k-i+1}\cdots s_{n-d_k-i+(l_i-1)};$$
	\item Putting them together we obtain $$w=w_0w_1\cdots w_{n-d_k-1}=s_{i_{l(w)}}\cdots s_{i_{1}}.$$
\end{enumerate} 
From now on we will always use precisely this reduced word for $w$.

\begin{example}
For the long Weyl element $w_{1,\ldots,1}$ the procedure looks as follows. We have $l_i = i+1$ for $i=0,\ldots,n-2$. The words attached to each row are precisely
\begin{equation}
	w_i = s_{n-i-1}\cdots s_{n-1}.\nonumber
\end{equation}
Thus we end up with the reduced word
\begin{equation}
	w_{1,\ldots,1} = s_{n-1}\cdot s_{n-2}s_{n-1}\cdot \ldots \cdot s_1\cdots s_{n-1}.\nonumber
\end{equation}
Similarly for the Weyl element $w_{1,n-2,1}=w_{\ast}$, which turned out to be of special interest in the arguments of \cite{Bl, AB}, the procedure gives the reduced word
\begin{equation}
	w_{\ast}=s_{n-1}s_{n-2}\cdots s_2s_1s_2\cdots s_{n-1}.\nonumber
\end{equation}
\end{example}

Next we enumerate the positive roots in $R(w^{-1})$ by requiring that
\begin{equation}
	\alpha_{i_j}([s_{i_{j-1}}\cdots s_{i_1}].\chi) = \beta_j(\chi).\nonumber
\end{equation}

Having made these arrangements we will also need some useful notation. Given $S  \subseteq R(w^{-1})$ we write 
\begin{equation}
	w_S =  s_{i_{j_{\sharp S}}}\cdots s_{i_{j_1}},\nonumber
\end{equation}  
for $R(w^{-1})\setminus S = \{\beta_{j_1},\ldots, \beta_{j_{\sharp S}} \}$ with $1\leq i_1< \ldots < i_{\sharp S}\leq l(w)$.
Note that one might have $w_{S} = w_{S'}$ for different $S,S'\subseteq R(w^{-1})$. But we must always have that $l(w)-\sharp S-l(w_S)$ is even. We can further decompose $S = S_{\nearrow}\sqcup  S_{\searrow}$ so that $\beta_j\in S_{\nearrow}$ if $l(s_{i_j}u)>l(u)$ with $u= s_{i_{j_{\ast}}}\cdots s_{i_{j_1}}$ where $j>j_{\ast}$ is the biggest index appearing in $R(w^{-1})\setminus S$. Write $R(w^{-1})\setminus S = S^{\bot}$. We are now ready to write down a general expression for $\mathcal{A}_w^{\chi}\phi_1^{\chi}$:

\begin{lemma}\label{general_expr}
We have
\begin{equation}
	\mathcal{A}_w^{\chi} \phi_1^{\chi} = \sum_{S\subseteq R(w^{-1})} l^{\frac{\sharp S-l(w)-l(w_S)}{2}}\cdot d_S \cdot \chi(\beta_S) \cdot \phi_{w_S}^{w.\chi},\label{gen_expr}
\end{equation}
where 
\begin{equation}
	d_S = \prod_{\beta\in S}\frac{1-l^{-1}}{1-\beta(\chi)} \nonumber
\end{equation}
and 
\begin{equation}
	\beta_S = \sum_{\beta\in S_{\nearrow}} \beta.\nonumber
\end{equation}
\end{lemma}
\begin{proof}
The proof proceeds by induction over the length of $w$. If $w$ is the identity, then $l(w)=0$ and we simply have $\mathcal{A}_w^{\chi} \phi_1 = \phi_1$. Thus, we suppose that $l(w)\geq 1$. In this case we write $w=w's$ for a simple reflection $s$ and $l(w') = l(w)-1$. Given that \eqref{gen_expr} holds for $w'$ we obtain
\begin{equation}
	\mathcal{A}_w^{\chi} \phi_1^{\chi} = \mathcal{A}_s^{w'\chi}\mathcal{A}_{w'}^{\chi} \phi_1^{\chi} = \sum_{S'\subseteq R(w'^{-1})} l^{\frac{\sharp S'-l(w')-l(w_{S'}')}{2}}\cdot d_{S'} \cdot \chi(\beta_{S'}) \cdot \mathcal{A}_s^{w'\chi}\phi_{w_{S'}'}^{w'.\chi}. \nonumber
\end{equation}
Given $S'\subseteq R(w'^{-1})$ we now compute $\mathcal{A}_s^{w'\chi}\phi_{w_{S'}'}^{w'.\chi}$ using \eqref{rec_intertwiners}. We have to consider two cases.

First, suppose that $sw_{S'}'>w_{S'}'$. We obtain
\begin{multline}
	l^{\frac{\sharp S'-l(w')-l(w_{S'}')}{2}}\cdot d_{S'} \cdot \chi(\beta_{S'}) \cdot \mathcal{A}_s^{w'\chi}\phi_{w_{S'}'}^{w'.\chi} \\ = l^{\frac{\sharp S'-l(w')-l(w_{S'}')}{2}}\cdot d_{S'} \cdot \chi(\beta_{S'}) \alpha(\chi)c_{\alpha}(\chi)\phi_{w_{S'}'}^{w.\chi} + l^{\frac{\sharp S'-l(w')-l(w_{S'}')-2}{2}}\cdot d_{S'} \cdot \chi(\beta_{S'})\phi_{sw_{S'}'}^{w.\chi},\nonumber
\end{multline}
for $R(w^{-1})\setminus R(w'^{-1}) = \{ \alpha\}$. In the notation set up above we  have $sw_{S'}' = w_{S'}$ and $w_{S'}' = w_S$ for $S=S'\cup \{\alpha\}$. Also note that, since $sw_{S'}'>w_{S'}$, we have $\alpha\in S_{\nearrow}$. Thus, we can rewrite the expression above as 
\begin{equation}
	l^{\frac{\sharp S'-l(w')-l(w_{S'}')}{2}}\cdot d_{S'} \cdot \chi(\beta_{S'}) \cdot \mathcal{A}_s^{w'\chi}\phi_{w_{S'}'}^{w'.\chi} \\ = l^{\frac{\sharp S'-l(w)-l(w_{S'})}{2}}\cdot d_{S'} \cdot \chi(\beta_{S'})\phi_{w_{S'}}^{w.\chi} + l^{\frac{\sharp S -l(w)-l(w_{S})-2}{2}}\cdot d_{S} \cdot \chi(\beta_{S})\phi_{w_{S}}^{w.\chi},
\end{equation}
which is exactly what we need.

The second case, when $sw_{S'}<w_{S'}$ is similar and we omit the details.
\end{proof}

Evaluating $\mathcal{A}_w^{\chi}\phi_1$ at $\mathbf{1}_n$ kills all terms with $w_S\neq 1$ in the expression \eqref{gen_expr}. Thus the proof is reduced to the purely combinatorial problem of studying the structure of $S\subseteq R(w^{-1})$ with $w_{S}=1$:
\begin{equation}
	[\mathcal{A}_w^{\chi} \phi_1](\mathbf{1}_n) = \sum_{\substack{S\subseteq R(w^{-1}),\\ w_S=1}} l^{\frac{\sharp S-l(w)}{2}}\cdot d_S \cdot \chi(\beta_S).\label{eqqqq}
\end{equation} 
For each such subset $S$ we have to show that $l^{\frac{\sharp S-l(w)}{2}}\cdot \chi(\beta_S)$ is not to big. An additional complication is that we have to do so for unramified characters $\chi = \chi_{\mathbf{s}}^{\rm ad}$ that are not dominant (i.e. they do not see all simple roots). We first convince ourselves that for such \textit{singular} characters the denominators $d_S$ are still well behaved:

\begin{lemma}\label{no_pole}
For $\chi_{\mathbf{s}}^{\rm ad}$ as above and $S\subseteq R(w^{-1})$ the denominators in $d_S$ are non-zero and we have $d_S\leq 1$.
\end{lemma}
\begin{proof}
We first observe that
\begin{equation}
	\alpha_i(\chi_{\mathbf{s}}^{\rm ad}) = \begin{cases}
		l^{-s_i} & \text{ if }i = d_1+\ldots + d_i,\\
		1 &\text{ else.}
	\end{cases} \nonumber
\end{equation}	
Because the simple roots $\alpha_i$ with $i=d_1,\ldots, d_1+\ldots+d_{k-1}$ are precisely the simple roots appearing in $R(w^{-1})$ (resp. $N_w$) we see that for every $\beta\in R(w^{-1})$ we have $\vert 1-\beta(\chi)\vert \geq  1-l^{-1}$. This immediately implies both claims.
\end{proof}

In order to make the weight $\chi_{\mathbf{s}}^{\rm ad}(\beta_S)$ more accessible we define
\begin{equation}
	v_{\rm ad}( \gamma ) = \sum_{i=1}^{k-1} k_{d_1+\ldots +d_i}.\nonumber
\end{equation}
for a weight $\gamma = \sum_{i=1}^{n-1} k_i\alpha_i$. In particular, if $\beta_{j_i}\in S_{\nearrow}$, then it contributes the weight 
\begin{equation}
	w(j_i) = \chi(\beta_{j_i}) = l^{-v_{\rm ad}(\beta_{j_i})} \nonumber
\end{equation}
to the corresponding term in \eqref{gen_expr}.

\begin{example}
If $w=w_{d_1,d_2}$ it is clear that
\begin{equation}
	v_{\rm ad}(\beta)= 1 \nonumber
\end{equation}
for all $\beta\in R(w^{-1})$. Thus we get the estimate
\begin{equation}
	[\mathcal{A}_{w_{d_1,d_2}}^{\chi_{\mathbf{s}^{\rm ad}}} \phi_1](\mathbf{1}_n) \leq  \sum_{\substack{S\subseteq R(w^{-1}),\\ w_S=1}} l^{\frac{\sharp S-l(w)}{2}-\sharp S_{\nearrow}},\nonumber
\end{equation}
when $\Re(\mathbf{s}) = \mathbf{1}+\boldsymbol{\epsilon}$. If $w_S=1$, then the set $S$ must contain at least $\lceil \frac{n-1}{2}\rceil$ elements. Indeed, for each simple root $s_i$ that occurs with odd multiplicity in $w_{d_1,d_2}$ a positive root $\beta_j$ with $s_i=s_{i_j}$ must be contained in $S$. Similarly one sees that $\sharp S_{\nearrow}\geq \lceil \frac{n}{2}\rceil.$ Using the obvious bound $\sharp S\leq l(w)$ get the estimate
\begin{equation}
	[\mathcal{A}_{w_{d_1,d_2}}^{\chi_{\mathbf{s}^{\rm ad}}} \phi_1](\mathbf{1}_n) \ll l^{-\frac{n}{2}} = \mathcal{V}_l^{\frac{1}{n-1}}
\end{equation} 
as predicted in Proposition~\ref{prop:crucial}. 

Of course in this simple case this bound is very crude in general. This can be seen by taking a closer look at the simplest cases. These are of course the Voronoi elements, which are given by $\{d_1,d_2\} = \{1,n-1\}$. For these $S= R(w^{-1})=S_{\nearrow}$ is the only choice for $S$ with $w_S=1$. We thus have
\begin{equation}
	[\mathcal{A}_{w}^{\chi_{\mathbf{s}}^{\rm ad}} \phi_1](\mathbf{1}_n) = d_{R(w^{-1})}l^{-s_1(n-1)}  \text{, for }w=w_{1,n-1},w_{n-1,1} \text{ and }\mathbf{s}=(s_1).  \nonumber
\end{equation}

Note that all admissible Weyl elements with at least one large block (in particular such elements as $w_{d_1,d_2}$) can be excluded using congruence conditions and admissibility constraints as in \cite{Bl, AB}. 
\end{example}

%\begin{remark}
%Note that this proves Proposition~\ref{prop:crucial} for $w=w_{1,n-1}$. (Note that the estimate is even stronger.) In general the rank $2$ (i.e. $r=2$) Weyl elements are easily treated. Indeed, for $w=w_{d_1,d_2}$, we observe that $\beta_j(\chi)$ is independent of $j$. (We even have $\vert\beta_j(\chi)\vert \leq l^{-1}$.) The situation is now illustrated by the following table
%\begin{center}
%	\begin{tabular}{ | c | c | c | c | c | c | c | c | c | c | }
%		\hline
%		 & $s_1$ & $s_2$ & $\dots$ & $s_{d_1-1}$ & $s_{d_1}$ & $s_{d_1+1}$ & $\ldots$ & $s_{n-2}$ & $s_{n-1}$\\ \hline
%		 $R_0$ & &&&& $1$ & $1$ &$\dots$& $1$ & $1$  \\\hline
%		 $R_1$ & & & & 1 & 1 & 1 & $\dots$ &  1&  \\ \hline
%		 $\vdots $ & & & $\Ddots$ & $\Ddots $ & $\Ddots$ & $\Ddots$ & $\Ddots$ & & \\ \hline
%		 $R_{d_1-2}$ &  & 1 & $\dots$ & 1 & 1 & 1 & & & \\\hline
%		 $R_{d_1-1}$ & 1 & 1 & $\dots$ & 1 & 1 & & & & \\ 
%		\hline 
%	\end{tabular}
%\end{center}
%The entry in the table represents $v_{\rm ad}(\beta_i)$. We now easily observe that no matter how we choose $S\subseteq R(w^{-1})$ with $w_S = 1$ every simple reflection contributes at least $l^{-1}$. Thus we have $\mathcal{A}_{w_{d_1,d_2}}^{\chi}\phi_1^{\chi}(1) \leq l^{-(n-1)}\leq l^{-\frac{n}{2}}$.

%Alternatively Weyl elements with big blocks such as $w_{d_1,d_2}$ can be handled classically by simply noting that admissible tuples $c$ for which the local orbital integral is non-zero (i.e. the ramified Kloosterman set is non-empty) must satisfy strong divisibility conditions. See \cite{Bl,AB} for related analysis along these lines.
%\end{remark}

In general, the weights $w(j_i)$ may vary, but we can still make the following observation:
\begin{equation}
	s_{i_j} = s_{i_{j'}} \implies v_{\rm ad}(\beta_j) = v_{\rm ad}(\beta_{j'}).\nonumber
\end{equation}
In particular, if $s_{i_j} = s_{i_{j'}}$, then $\omega(i_j) = \omega(i_{j'})$. Thus it makes sense to write
\begin{equation}
	\omega(s_i) = \omega(i_j) \text{ for some $j$ with $s_i=s_{i_j}$.} \nonumber
\end{equation}
We now pick any $S\subseteq R(w^{-1})$ with $w_S=1$. Let us write $$m_S(s_i) = \sharp\{ \beta_j\in S\colon s_{i_j}=s \}$$ for the multiplicity of the simple reflection $s$ in $S$. Since we require $w_S=1$ we must have 
\begin{equation}
	m_{R(w^{-1})}-m_S(s_i) \equiv 0 \text{ mod }2.\nonumber
\end{equation}
The largest contribution to \eqref{eqqqq} can be bounded by the contribution of the critical set $S_{\rm cri}$ which we will now define. It is the smallest subset $S\subseteq R(w^{-1})$ for which $w_S=1$ could hold and is characterized by $S_{\rm cri}=S_{\nearrow}$ and 
\begin{equation}
	m_{S_{\rm cri}}(s_i) = \begin{cases} 
	1 &\text{ if }m_{R(w^{-1})}(s_i) \text{ is odd,} \\
	0 & \text{ else.}
\end{cases} \nonumber
\end{equation}
Since 
\begin{equation}
	[\mathcal{A}_{w}^{\chi_{\mathbf{s}}^{\rm ad}}\phi_1^{\chi_{\mathbf{s}}^{\rm ad}}](\mathbf{1}_n) \ll l^{\frac{\sharp S_{\rm cri}-l(w)}{2}} \cdot \vert\chi(\beta_{S_{\rm cri}})\vert \ll  l^{\frac{\sharp S_{\rm cri}-l(w)}{2}} \prod_{m_{R(w^{-1})}(s_i) \equiv 1 \text{ mod }2}\omega(s_i) \nonumber
\end{equation}
it remains to compute these contributions. This is an elementary but tedious task which can be carried out using the block structure $(d_1,\ldots,d_k)$ of $w$ and our choice of reduced word constructed above. In order to establish Proposition~\ref{prop:crucial} one needs to show
\begin{equation}
	\sharp S_{\rm cri} - l(w) -2\sum_{\beta\in S_{\rm cri}}v_{\rm ad}(\beta) \geq -(k-1)n. \label{to_show}
\end{equation}

%We are now ready to prove Proposition~\ref{prop:crucial}. We will do so for the critical case $w=w_{1,\ldots,1}$. The treatment of general elements $w=w_{d_1,\ldots, d_r}$ is similar only notationally even more involved. We start by looking at the following table

%\begin{center}
%	\begin{tabular}{ | c | c | c | c | c | c | c | }
%		\hline
%	 & $s_1$ & $s_2$ & $ \dots$ & $s_{n-3}$ & $s_{n-2}$ & $s_{n-1}$ \\ \hline
%	  $ R_0$ & & & & & & $1$ \\ \hline
%	  $ R_1$ & & & & & $2$ & $1$ \\ \hline
%	  $ R_2$ & & & & $ 3 $ & $ 2 $ & $ 1$ \\ \hline
%	  $\vdots$ & & & $\Ddots$& $\vdots$ & $\vdots$ & $\vdots$ \\ \hline
%	  $R_{n-3}$ & & $n-2$& $\dots$ & $3$ & $2$ & $1$ \\\hline
%	  $R_{n-2}$ & $n-1$& $n-2$& $\dots$ & $3$ & $2$& $1$ \\
%	  \hline 
%\end{tabular}
%\end{center}

%Here we have indicated how th simple reflection appears in our reduced word for $w=w_{1,\ldots,1}$ organised according to the rows. The weight in the table is again $v_{\rm ad}(\beta_j)$ for the corresponding index $j$. 

For the long element $w=w_{1,\ldots,1}$ we compute 
\begin{equation}
	2v_{\rm ad}(\beta_{S_{\rm cri}}) = 2 \cdot \sum_{\substack{i=1,\\ \text{odd}}}^{n-1} (n-1) = \begin{cases}
		\frac{n^2-1}{2} & \text{ if $n-1$ is even,}\\
		\frac{n^2}{2}& \text{ if $n-1$ is odd.}
	\end{cases}
\end{equation}
On the other hand, we obviously have
\begin{multline}
	l(w)-\sharp S_{\rm cri} = l(w)-\sharp \{ i\colon m_{R(w^{-1})}(s_i) = i \equiv 1 \text{ mod }2 \} \\ = \frac{n(n-1)}{2}-\lceil \frac{n-1}{2}\rceil = \begin{cases}
		\frac{(n-1)^2}{2} &\text{ if $n-1$ is even,} \\
		\frac{n(n-2)}{2} & \text{ if $n-1$ is odd.}
	\end{cases} \nonumber 
\end{multline}
Adding these two contributions shows that
\begin{equation}
	l(w)- \sharp S_{\rm cri} + 2v_{\rm ad}(\beta_S) = n(n-1),
\end{equation}
which is \eqref{to_show} on the nose.

The proof of \eqref{to_show} and thus Proposition~\ref{prop:crucial} for general admissible Weyl elements $w=w_{d_1,\ldots,d_k}$ is similarly elementary but notationally more involved. One essentially arranges $R(w^{-1})$ into $d_i\times d_j$ blocks (with $i<j$) on which $\omega(\ast)$ is constant. The analysis is then similar to the one of the long element with $n$ replaced by $k$. We omit the details.

\section{Archimedean Preliminaries}

We first recall some basics concerning the spherical Whittaker function at $\infty$. Given $\mu=(\mu_1,\ldots,\mu_n)\in \mathbb{C}^n$ let $W_{\mu}\colon \mathbb{R}_{>0}^{n-1}\to \mathbb{C}$ denote the standard spherical Whittaker function defined in, for example, \cite[2.16]{AB}. Further, as in \cite[2.8]{AB}, we introduce the inner product
\begin{equation}
	\langle f,g\rangle_{\tilde{T}} = \int_{\mathbb{R}_{>0}^{n-1}}f(y)\overline{g(y)}d^{\ast}y, \text{ for }f,g\colon \mathbb{R}_{>0}^{n-1}\to \mathbb{C}.\nonumber
\end{equation}
Given a function $E\colon \mathbb{R}_{>0}^{n-1}\to \mathbb{C}$ and a parameter $X=(X_1,\ldots,X_{n-1})\in \mathbb{R}_{>0}^{n-1}$ we follow \cite[(5.6)]{Bl} and set
\begin{equation}
	E^{(X)}(y_1,\ldots,y_{n-1}) = E(X_1y_1,\ldots, X_{n-1}y_{n-1}).\nonumber
\end{equation}
We are now ready to recall \cite[Lemma~3.2]{AB}, which is an extension of \cite[Lemma~5]{Bl}:

\begin{lemma}\label{arch_test_fct}
Let $M\geq 2$ and $Z>M^{K^2}$. Then there exists $r\in \mathbb{N}$, $r\leq M^K$ and a collection of (measurable) functions $E_1,\ldots,E_r\colon [M^{-K},M^K]^{n-1}\to [0,1]$ such that
\begin{equation}
	\sum_{j=1}^{r}\vert \langle E_j^{(Z,1,\ldots,1)},W_{\mu}\rangle_{\tilde{T}}\vert^2 \gg Z^{2\eta_1+2\sigma(\mu)}M^{-K}, \nonumber
\end{equation}
for $\mu$ satisfying $\Vert \mu\Vert \leq M$ and $\sigma(\mu) = \max_{i=1,\ldots,n}\vert \Re(\mu_i)\vert<\frac{1}{2}$. (Recall that $\eta_1=\frac{n-1}{2}$ is the first component of $\eta$ defined in \eqref{eta}.)
\end{lemma}

We now turn towards some important preliminaries needed for the estimation of orbital integrals. First, we extend the map ${\rm y}\colon \tilde{T}(\mathbb{R}_+)\to \mathbb{R}_{>0}^{n-1}$ to $G(\mathbb{R})$ by setting ${\rm y}(xyk\alpha) = {\rm y}(y)$ for $x\in U(\mathbb{R})$, $y\in\tilde{T}(\mathbb{R}_{>0})$, $k\in {\rm O}_n(\mathbb{R})$ and $\alpha\in Z_+$. We further define
\begin{equation}
	s(i,j) = \frac{1}{n}\begin{cases}
		 i(n-j) &\text{ if }i\leq j,\\
		 j(n-i) &\text{ if }i>j,
	\end{cases} \nonumber
\end{equation}
as in \cite[(3.2)]{Bl}. This numbers will appear as exponents through \cite[Lemma~1]{Bl}, which will be used below multiple times.

Let $g\in{\rm GL}_n(\mathbb{R})$ and $1\leq j\leq n$, as in \cite{Bl} we write $\Delta_j(g)$ for the volume of the parallelepiped spanned by the last $j$ rows of $g$. In particular, we have
\begin{equation}
	g=u\cdot {\rm diag}(\Delta_n/\Delta_{n-1},\Delta_{n-1}/\Delta_{n-2},\ldots,\Delta_1)\cdot k \text{ for  }u\in U(\mathbb{R}) \text{ and }k\in O_n(\mathbb{R}).\nonumber
\end{equation}
We define the function
\begin{equation}
	f_{1}(g) = \begin{cases}
		1 &\text{ if }\Delta_j\leq  1 \text{ for  all }1\leq j\leq n-1,\nonumber \\
		0 &\text{ else.}
	\end{cases} \nonumber
\end{equation}
Note that $f_1(u\cdot g) = f_1(g)$. With this in mind we set
\begin{equation}
	F_{\mathbf{s}}(g) = \int_{\mathbb{R}_{>0}^{n-1}}f_1(c^{\ast}g)\cdot [\delta_B^{-\frac{1}{2}}\chi_{\mathbf{s}}](c^{\ast}) d^{\times}c.\nonumber
\end{equation}
For $\Re(\mathbf{s})>-\mathbf{1}$ the integral converges and defines an element in $I(\chi_{-\mathbf{s}})$. Furthermore, one checks that $F_{\mathbf{s}}(g) = c(\mathbf{s})\phi_{\circ}^{\chi_{\mathbf{s}}}$, where $\phi_{\circ}^{\chi_{-\mathbf{s}}}\in I(\chi_{-\mathbf{s}})$ is the unique spherical element (normalized by $\phi_{\circ}^{\chi_{-\mathbf{s}}}(1)=1$) and $$c(\mathbf{s}) = \prod_{i=1}^{n-1}\frac{1}{s_i-1}.$$ The archimedean intertwining operator $\mathcal{A}_w^{\chi}\colon I(\chi) \to I(\chi^w)$ is defined by
\begin{equation}
	[\mathcal{A}_w^{\chi}\Phi](g) = \int_{U_w(\mathbb{R})} \Phi(wxg)dx.\nonumber
\end{equation}
This converges for $\chi$ sufficiently dominant (i.e. $\Re(\mathbf{s})>\mathbf{1}$) and is understood by analytic continuation otherwise. (See \cite[10.1.11]{Wa}.)

\begin{lemma}\label{arch_prep_lem}
For $w\in W$ and $\Re(s_i) \geq 1+\epsilon$ we have
\begin{equation}
	\int_{\mathbb{R}^{n-1}}\int_{U_w(\mathbb{R})} f_1(c^{\ast}wx)dx\cdot [\delta_B^{-\frac{1}{2}}\chi_{\mathbf{s}}](c^{\ast})  d^{\times}c = [\mathcal{A}_w F_{\mathbf{s}}](\mathbf{1}_n) = c(\mathbf{s})\cdot \prod_{\alpha\in R(w^{-1})}\frac{\Gamma_{\mathbb{R}}(\alpha(\chi_{-\mathbf{s}}))}{\Gamma_{\mathbb{R}}(\alpha(\chi_{-\mathbf{s}})+1)}, \nonumber
\end{equation}
where $\Gamma_{\mathbb{R}}(s) = \pi^{-s}\Gamma(\frac{s}{2})$ is the typical archimedean $L$-factor and $\alpha(\chi_{-\mathbf{s}})$ is defined by $\alpha_{i}(\chi_{-\mathbf{s}}) = s_i$ on simple roots and extended additively. Note the similarity to the $p$-adic analogue \eqref{nonarch_vers_comp}.
\end{lemma}
\begin{proof}
This is a consequence of the Gindikin-Karpelevich formula. See for example \cite[Section~3]{Lan} where the archimedean version is stated.
\end{proof}

With this at hand we can  generalize \cite[Lemma~3]{Bl} to all Weyl elements. The proof will rely on basic estimates for (archimedean) intertwining operators and is very similar to  the argument used in \cite{DR} to estimate the size of ($p$-adic) Kloosterman sets. Indeed, the estimate can be seen as an archimedean analogue of the latter.

\begin{lemma}\label{lm:gen_lm3_bl}
Let $B\in \mathbb{R}_{>0}^{n-1}$ and $w\in W$. Then we have
\begin{equation}
	{\rm vol}\{x\in U_w(\mathbb{R})\colon \Delta_j(wx)\leq B_j \text{ for }1\leq j\leq n-1\} \ll (B_1\cdots B_{n-1})^{1+\epsilon} \nonumber
\end{equation}
for any $\epsilon>0$.
\end{lemma}

\begin{proof}
We start by observing that
\begin{multline}
	{\rm Vol}\{x\in U_w(\mathbb{R})\colon \Delta_j(wx)\leq B_j \text{ for }1\leq j\leq n-1\} \\ = {\rm Vol}\{x\in U_w(\mathbb{R})\colon \Delta_j((B^{\ast})^{-1}wx)\leq 1 \text{ for }1\leq j\leq n-1\} \\
	= \int_{U_w(\mathbb{R})}f_1((B^{\ast})^{-1}wx)dx.\nonumber
\end{multline}
By Mellin inversion we get
\begin{equation}
	{\rm Vol}\{x\in U_w(\mathbb{R})\colon \Delta_j(wx)\leq B_j \text{ for }1\leq j\leq n-1\} \\
	=  \frac{1}{(2\pi i)^{n-1}}\int_{\boldsymbol{\sigma}}[\mathcal{A}_wF_{\boldsymbol{s}-\mathbf{1}}](1)B^{\mathbf{s}}d\mathbf{s}.\nonumber
\end{equation}
Note that the factor $\delta_B^{-\frac{1}{2}}(c^{\ast})$ appearing in the definition of $F_{\mathbf{s}}$ is accounted for by the shift in $\mathbf{s}$. Inserting Lemma~\ref{arch_prep_lem}, shifting the contour(s) to $\boldsymbol{\sigma}=\mathbf{1}+\boldsymbol{\epsilon}$ and estimating trivially yields the desired result.
\end{proof}

This estimate is crucial in order to control the archimedean orbital integrals. These are defined as
\begin{equation}
	\mathcal{O}_{\infty}^{\boldsymbol{\psi}_{\infty}}(tw,F;y) = \int_{U_w(\mathbb{R})} F(twuy)\boldsymbol{\psi}_{\infty}^{-1}(u)du,\label{arch_orb}
\end{equation}
for $F(xyk\alpha) = \boldsymbol{\psi}_{\infty}(x)E({\rm y}(y)),$ if $x\in U(\mathbb{R})$, $y\in \tilde{T}(\mathbb{R}_+)$, $k\in {\rm O}_n(\mathbb{R})$ and $\alpha\in Z_+$. We will conclude this section by deriving a template estimate which can be used as a black box to bound the contributions of the archimedean orbital integrals on the geometric side of our pre-Kuznetsov formula. As the proof will show, this is nothing but a re-packaging of the corresponding estimates in \cite{Bl, AB}. Note however that in loc. cit this estimate was only carried out for the special Weyl element $w_{\ast}$. Here we can treat all Weyl elements and this will also be necessary for our application.

\begin{lemma}\label{arch_lem}
Fix a measurable function $E\colon \mathbb{R}_{>0}^{n-1}\to \mathbb{C}$ supported in $[M^{-K},M^K]$ and let $X=(Z,1,\ldots,1)\in \mathbb{R}_{>0}^{n-1}$ with $Z\geq 1$ be a parameter. Then, for $w=w_{d_1,\ldots,d_k}$ admissible and $c\in\mathbb{R}_{>0}^{n-1}$ we have
\begin{equation}
	\int_{\tilde{T}(\mathbb{R}_{>0})}\mathcal{O}_{\infty}^{\boldsymbol{\psi}_{\infty}}(c^{\ast}w,E^{(X)};y)\overline{E^{(X)}({\rm y}(y))}d^{\ast}y \ll M^{K'}\frac{Z^{2\eta_1+\epsilon}}{c_1\cdots c_{n-1}}
\end{equation}
Furthermore, the integral vanishes unless $c_i \leq M^{K'}Z $. Here $K'$ depends only on $K$ and we assume $M\geq 2$.
\end{lemma}
\begin{proof}
We first slightly rewrite the integrals. We define
\begin{equation}
	A=\iota(X)c^{\ast}w\iota(X)^{-1}w^{-1} = \iota(X\cdot {}^w(X))c^{\ast}\in T_0(\mathbb{R}),\nonumber
\end{equation}	
as in \cite[(6.3)]{Bl}. Changing variables $\iota(X)y \to y$ yields
\begin{multline}
	\int_{\tilde{T}(\mathbb{R}_{>0})}\mathcal{O}_{\infty}^{\boldsymbol{\psi}_{\infty}}(c^{\ast}w,E^{(X)};y)\overline{E^{(X)}({\rm y}(y))}d^{\ast}y \\ = \frac{Z^{2\eta_1}}{c_1\cdots c_{n-1}\cdot y(A)^{\eta}}	\int_{\tilde{T}(\mathbb{R}_{>0})}\mathcal{O}_{\infty}^{\boldsymbol{\psi}_{\infty}}(\iota(X)^{-1}Aw,E^{(X)};y)\overline{E({\rm y}(y))}d^{\ast}y.\nonumber
\end{multline}	
Next, note that due to the support of $E$ we can restrict the $y$-integral to the range $${\rm y(y)}\in [M^{-K},M^K]^{n-1}.$$ Now applying \cite[Lemma~1]{Bl} as in \cite[Section~7]{Bl} (see also the proof of \cite[Lemma~5.1]{AB}) we find that
\begin{equation}
	c_j \ll (1+2M^K)^K\cdot \begin{cases}
		Z &\text{ for }j\leq n-d_1,\\
		1&\text{ for } j>n-d_1.
	\end{cases}
\end{equation}

We turn towards estimating the integral. This is done trivially using \cite[Lemma~1]{Bl} together with Lemma~\ref{lm:gen_lm3_bl} (which generalizes \cite[Lemma~3]{Bl}). More precisely, we have
\begin{align}
	&\int_{\tilde{T}(\mathbb{R}_{>0})}\mathcal{O}_{\infty}^{\boldsymbol{\psi}_{\infty}}(\iota(X)^{-1}Aw,E^{(X)};y)\overline{E({\rm y}(y))}d^{\ast}y \nonumber\\
	&\ll \int_{\tilde{T}(\mathbb{R}_{>0})}\int_{U_w(\mathbb{R})}\vert E({\rm y}(Awxy))\cdot E({\rm y}(y))\vert dxd^{\ast}y \nonumber\\
	&\ll_E {\rm Vol}\{x\in U_w(\mathbb{R})\colon \Delta_j(w x)\ll_E \prod_{[i=1}^{n-1}{\rm y}(A)_i^{s(ij)},\, 1\leq j\leq n-1 \} \nonumber\\
	&\ll_E \prod_{i=1}^{n-1}\prod_{j=1}^{n-1}{\rm y}(A)_i^{s(i,j)(1+\epsilon)} = {\rm y}(A)^{\eta(1+\epsilon)}.\nonumber
\end{align}
Inserting this above completes the proof.
\end{proof}

\section{On the spectral and geometric analysis of Poincar\'e series}

In this section, we develop a (pre)-Kuznetsov formula comparable to \cite[Lemma~6]{Bl}. Here we chose to work in the adelic framework closely following \cite{Ste}.

\subsection{Definition and Basic Properties}

We call $F\colon G(\mathbb{A})\to \mathbb{C}$ an decomposable $\boldsymbol{\psi}$-Whittaker function if it is of the form $F(g)=\prod_v F_v(g_v)$ and satisfies
\begin{equation}
	F(uzg) = \boldsymbol{\psi}(u)F(g), \text{ for } u\in U(\mathbb{A}) \text{ and }z\in Z(\mathbb{A}) \nonumber
\end{equation}
This definition only slightly differs from \cite[Definition~1.5]{Ste}, where also a $K$-type is taken into account.

Given a $\boldsymbol{\psi}$-Whittaker function we define the corresponding (deformed) Poincar\'e series by
\begin{equation}
	P_F(g;\lambda) = \sum_{Z(\mathbb{Q})U(\mathbb{Q})\backslash {\rm GL}_n(\mathbb{Q})} F(\gamma g)\cdot \Vert \gamma g\Vert_{\mathbb{A}}^{\lambda}.\nonumber
\end{equation}
This will converge absolutely for sufficiently dominant $\lambda$. For our purposes we make the following specific choice:
\begin{itemize}
	\item We define $F_{\infty}$ as in \cite{Bl} as follows. For a fixed compactly supported function $E\colon \mathbb{R}_{>0}^{n-1}\to \mathbb{C}$ we set $$F_{\infty}(xyk\alpha) = \boldsymbol{\psi}_{\infty}(x)E({\rm y}(y)),$$ for $x\in U(\mathbb{R})$, $y\in \tilde{T}(\mathbb{R}_+)$, $k\in {\rm O}_n(\mathbb{R})$ and $\alpha\in Z_+$.
	\item Given $T\in \mathcal{H}_W$ we define 
	\begin{equation}
		F_l(g_l) = \begin{cases}
			\boldsymbol{\psi}_l(u_l)T(k_l) &\text{ if }g_l=u_lt_lk_l \text{ with }u_l\in U(\mathbb{Q}_l),\, t_l\in Z(\mathbb{Q}_l)T(\mathbb{Z}_l) \text{ and }k_l\in K_l,\\
			0 &\text{ else.}
		\end{cases}\nonumber
	\end{equation}
	\item At the remaining places $p\neq l$ we make the usual choice:
	\begin{equation}
			F_p(g_p) = \begin{cases}
			\boldsymbol{\psi}_p(u_p) &\text{ if }g_p=u_pt_pk_p \text{ with }u_p\in U(\mathbb{Q}_p),\, t_p\in Z(\mathbb{Q}_p)T(\mathbb{Z}_p) \text{ and }k_p\in K_p,\\
			0 &\text{ else.}
		\end{cases}\nonumber
	\end{equation}
\end{itemize}
For this choice we can work with $\lambda=0$. To keep  track of  $T$ and $E$ we write
\begin{equation}
	P(E,T;g) = P_F(g;0).\nonumber
\end{equation}
We first observe that for $S,T\in \mathcal{H}_W$ we have
\begin{equation}
	[SP(E,T;\cdot )](g) = P(E,S^{\vee}\ast T;g).\nonumber
\end{equation}

Let $\phi\in L^2(G(\mathbb{Q})\backslash G(\mathbb{A})^1)$ with associated Whittaker function
\begin{equation}
	W_{\phi}(g) = \int_{U(\mathbb{Q})\backslash U(\mathbb{A})} \phi(ug)\boldsymbol{\psi}(u)^{-1}du.\nonumber
\end{equation}
Furthermore, assume that $W_{\phi}$ is factorisable:
\begin{equation}
	W_{\phi}(g) = \prod_v W_{\phi,v}(g_v).\nonumber
\end{equation}

\begin{lemma}\label{lm:innerprodpoincare}
Let $\phi\in L^2(X_l)$ with factorisable Whittaker function. We have
\begin{equation}
	\langle P_F(\cdot;\lambda),\phi\rangle_2 =  2^ {n-1}C_n^{-1}\cdot W_{T\phi,{\rm fin}}(1)\cdot \langle [F_{\infty}\cdot \Vert\cdot \Vert_{\infty}^{\lambda} ]\circ{\rm y},W_{\phi,\infty}\circ{\rm y}\rangle_{\tilde{T}}\nonumber
\end{equation}
with $W_{T\phi,{\rm fin}} = \prod_{v<\infty}W_{T\phi,v}$ and where $C_n$ is the constant from \eqref{def_Cn}.
\end{lemma}
\begin{proof}
This is a standard computation using the unfolding trick. We follow for example the proof of \cite[Theorem~1.13]{Ste}. Inserting the definition of $P_F$, unfolding integral and sum and recognizing the Jacquet integral gives
\begin{equation}
	\langle P_F(\cdot;\lambda),\phi\rangle_2 = C_n^{-1}\cdot \int_{Z(\mathbb{A})U(\mathbb{A})\backslash G(\mathbb{A})} F(g)\overline{W_{\phi}(g)}\cdot \Vert g\Vert_{\mathbb{A}}^{\lambda}d_{\rm pr}g.\nonumber
\end{equation}
Here we used that the integrand is $Z(\mathbb{A})$ invariant and that ${\rm Vol}(Z_+Z(\mathbb{Q})\backslash Z(\mathbb{A}))=1$. The remaining integrals are factorisable and we can continue locally. 

For $v=\infty$ the (local) Whittaker vectors $W_{\phi,\infty}$ and $F_{\infty}$ are both spherical. Since $B(\mathbb{R})\cap{\rm SO}_n(\mathbb{R}) = 1$ and $[V(\mathbb{R})\colon Z(\pm 1)] = 2^{n-1}$ we have
\begin{equation}
	\int_{(Z(\mathbb{R})\cap V(\mathbb{R}))\backslash V(\mathbb{R})}\int_{(B(\mathbb{R})\cap K_{\infty})\backslash K_{\infty}}F_{\infty}(tvk)\overline{W_{\phi,\infty}(tvk)}dkdv = 2^{n-1}\cdot F_{\infty}(t)\overline{W_{\phi,\infty}(t)}.\nonumber
\end{equation}
In particular,
\begin{multline}
	\int_{Z_+\backslash T(\mathbb{R}_{\geq 0})}\int_{(Z(\mathbb{R})\cap V(\mathbb{R}))\backslash V(\mathbb{R})}\int_{(B(\mathbb{R})\cap K_{\infty})\backslash K_{\infty}}F_{\infty}(tvk)\overline{W_{\phi,\infty}(tvk)}dkdv \Vert t\Vert_{\infty}^{\lambda}dt \\ = 2^{n-1}\cdot \langle [F_{\infty}\cdot \Vert\cdot \Vert_{\infty}^{\lambda} ]\circ{\rm y},W_{\phi,\infty}\circ{\rm y}\rangle_{\tilde{T}}.\nonumber
\end{multline}
as expected. The place $v=l$ is treated by observing that
\begin{equation}
	\int_{Z(\mathbb{Q}_l)\backslash T(\mathbb{Q}_l)}\int_{(B(\mathbb{Q}_l)\cap K_{l})\backslash K_{l}}F_{l}(tk)\overline{W_{\phi,l}(tk)}dk \Vert t\Vert_{l}^{\lambda}dt = \int_{B(\mathbb{Z}_p)\backslash K_p}T(k)W_{\phi,l}(k)dk=W_{T\phi,l}(1).\nonumber
\end{equation}
At the remaining (finite) places $v=p\neq l$ we can argue as in \cite[Example~1.14]{Ste} to find
\begin{equation}
	\int_{Z(\mathbb{Q}_p)\backslash T(\mathbb{Q}_p)}\int_{(B(\mathbb{Q}_p)\cap K_{p})\backslash K_{p}}F_{p}(tk)\overline{W_{\phi,p}(tk)}dk \Vert t\Vert_{p}^{\lambda}dt = W_{\phi,p}(1).\nonumber
\end{equation}
\end{proof}

\subsection{A spectral bound}

In this section, we will give a spectral interpretation of the inner product of two Poincar\'e series:
\begin{equation}
	I(E,T) = \langle P(E,T;\ast),P(E,T;\ast)\rangle.\nonumber
\end{equation}
To do this we apply Parseval and use positivity to drop the non-cuspidal contribution. This gives the inequality
\begin{equation}
	\sum_{\pi\mid X_l}\sum_{\phi\in \mathcal{B}_l(\pi)} \frac{\vert \langle P(E,T;\ast),\phi\rangle\vert^2}{\Vert \phi\Vert^2} \leq I(E,T).\nonumber
\end{equation}
Applying Lemma~\ref{lm:innerprodpoincare} we obtain
\begin{equation}
	\sum_{\pi\mid X_l}\sum_{\varpi\in \mathcal{B}_q(\pi)} \vert \langle P(E,T;\ast),\phi\rangle\vert^2 = 4^{n-1}C_n^{-2}\sum_{\pi\mid X_l}\sum_{\phi\in \mathcal{B}_l(\pi)} \frac{\vert W_{T\phi,\rm fin}(1)\vert^2}{\Vert \phi\Vert^2} \cdot \vert  \langle F_{\infty}\circ{\rm y},W_{\phi,\infty}\circ{\rm y}\rangle_{\tilde{T}}\vert^2.\label{crude_exp}
\end{equation}
Let us record a slight reformulation of this in form of a lemma:
\begin{lemma}\label{spec_ineq}
Let $E\colon \mathbb{R}_{>0}^{n-1}\to \mathbb{C}$ be a fixed compactly supported function, $X\in \mathbb{R}_{>0}^{n-1}$ a parameter and $T\in \mathcal{H}_W$. Then we have
\begin{equation}
	4^{n-1}C_n^{-2}\sum_{\pi\mid X_l}\frac{J_{\pi_l}(T^{\vee}\ast T)}{l(\pi)}\cdot \vert\langle W_{\mu_{\infty}(\pi)},E^{(X)}\rangle_{\tilde{T}}\vert^2 \leq I(E^{(X)},T),\nonumber
\end{equation}
where
\begin{equation}
	l(\pi) = \lim_{s\to 1}\frac{L^S(s,\pi\times\tilde{\pi})}{\Delta_{{\rm GL}_n}(s)}\cdot \langle W_{\mu_{\infty}(\pi)},W_{\mu_{\infty}(\pi)}\rangle_{{\rm Wh}(\pi_{\infty})},\nonumber
\end{equation}
for $S=\{\infty,l\}$ and $\Delta_{{\rm GL}_n}^S(s) = \prod_{j=1}^n\zeta^S(s+j-1)$. (The superscript $S$ indicates the the local Euler-factors at the places in $S$ are omitted.)
\end{lemma}
\begin{proof}
First we observe that by considering a basis of factorisable forms we have a bijection between $\mathcal{B}_l(\pi)$ and $$B_l(\pi_{\infty})\times\prod_p \mathcal{B}_l(\pi_p),$$ where $\pi=\pi_{\infty}\otimes\bigotimes_p\pi_p.$ Further, we note that for $p\neq l$ the local representation $\pi_p$ is unramified and $\sharp\mathcal{B}_l(\pi_p) = 1$. We can choose the basis element so that $$W_{\phi,p}(\mathbf{1}_n)=1.$$ At the archimedean place the local basis contains only the spherical vector and we can make sure that
\begin{equation}
	W_{\phi,\infty}\vert_{\tilde{T}(\mathbb{R}_{>0})}\circ{\rm y} = W_{\mu_{\infty}(\pi)}.\nonumber 
\end{equation}
Combining these arrangements with the definition of $W_{\infty}$ we find that the $\phi$-sum in \eqref{crude_exp} equals:
\begin{equation}
	\sum_{\phi\in \mathcal{B}_l(\pi)} \frac{\vert W_{T\phi,l}(1)\vert^2}{\Vert \phi\Vert^2} \cdot \vert  \langle W_{\mu_{\infty}(\pi)},E^{(X)}\rangle_{\tilde{T}}\vert^2
\end{equation}

We now have to compute the $L^2$-norm of $\phi$ using the Rankin-Selberg method. This can be done as in \cite[Section~4]{LM}. We find that
\begin{equation}
	\frac{\vert W_{T\phi,l}(1)\vert^2}{\Vert \phi\Vert^2} = \lim_{s\to 1}\frac{\Delta_{{\rm GL}_n}(s)}{L^S(s,\pi\times\tilde{\pi})}\cdot \langle W_{\mu_{\infty}(\pi)},W_{\mu_{\infty}(\pi)}\rangle_{{\rm Wh}(\pi_{\infty})}^{-1} \cdot \frac{I_l(W_{T\phi,l})}{\langle W_{\phi,l},W_{\phi,l}\rangle_{{\rm Wh}(\pi_l)}}, \nonumber
\end{equation}
where the local integral $I_l(\cdot)$ was defined in \eqref{def_Il}. We therefore have
\begin{equation}
	\sum_{\phi\in \mathcal{B}_l(\pi)} \frac{\vert W_{T\phi,l}(1)\vert^2}{\Vert \phi\Vert^2} \cdot \vert  \langle W_{\mu_{\infty}(\pi)},E^{(X)}\rangle_{\tilde{T}}\vert^2 = \frac{1}{l(\pi)}\sum_{\phi\in \mathcal{B}_l(\pi)} \frac{I_l(W_{T\phi,l})}{\langle W_{\phi,l},W_{\phi,l}\rangle_{{\rm Wh}(\pi_l)}} \cdot \vert  \langle W_{\mu_{\infty}(\pi)},E^{(X)}\rangle_{\tilde{T}}\vert^2
\end{equation}
The sum over $\phi$ can be replaces by purely local sum over $B_l(\pi_l)$ and we recognize the Bessel distribution $J_{\pi_l}(T^{\vee}\ast T)$ by applying Lemma~\ref{recognice_bessel}.
\end{proof}

For later reference we recall the following upper bound on $l(\pi)$:

\begin{lemma}\label{upp_bound_lpi}
For $\pi\mid X_l$ we have
\begin{equation}
	l(\pi) \ll (\Vert \mu_{\infty}(\pi)\Vert\cdot  l)^{\epsilon}.\nonumber
\end{equation}
\end{lemma}
\begin{proof}
We first note that the archimedean inner product can be evaluate using Stade's formula. See \cite[Section~3]{AB} for details. One obtains
\begin{equation}
	\langle W_{\mu_{\infty}(\pi)},W_{\mu_{\infty}(\pi)}\rangle_{{\rm Wh}(\pi_{\infty})} \ll 1.
\end{equation}
Thus we obtain
\begin{equation}
	l(\pi) \ll l^{\epsilon}\vert L(1,\pi,{\rm Ad})\vert. \nonumber
\end{equation}
The result follows by bounding the adjoint $l$-function at the edge of the critical strip using \cite{Li}.
\end{proof}

\subsection{Geometric Considerations}

The geometric analysis of the inner product $I(E,T)$ of two Poincar\'e series is also standard. One starts by unfolding the inner product:
\begin{equation}
	I(E,T) = C_n^{-1}2^{n-1}\int_{\tilde{T}(\mathbb{R}_{>0})} \int_{U(\mathbb{Q})\backslash U(\mathbb{A})} P(E,T^{\vee}\ast T; xy)\overline{\boldsymbol{\psi}(x)}dx \cdot \overline{E({\rm y}(y))}d^{\ast}y.\nonumber
\end{equation}
We therefore have to compute the Fourier coefficient of $P(E,T^{\vee}\ast T;\cdot)$. For convenience we temporarily set $S=T^{\vee}\ast T$. This is done following \cite[Section~2]{Ste}. 

The Bruhat decomposition can be written as
\begin{equation}
	G(\mathbb{Q}) = \bigsqcup_{\substack{w\in W,\\ t\in T(\mathbb{Q})}} U(\mathbb{Q})twU_w(\mathbb{Q}).\nonumber 
\end{equation}
Opening the definition of the Poincar\'e series $P(E,S;\cdot)$ and write the Fourier coefficient as
\begin{equation}
	\int_{U(\mathbb{Q})\backslash U(\mathbb{A})} P(E,T^{\vee}\ast T; xy)\overline{\boldsymbol{\psi}(x)}dx = \sum_{\substack{w\in W,\\ t\in Z(\mathbb{Q})\backslash T(\mathbb{Q})}} \int_{U(\mathbb{Q})\backslash U(\mathbb{A})} \sum_{\gamma\in U_w(\mathbb{Q})} F(tw\gamma xy)\boldsymbol{\psi}^{-1}(x)dx. \nonumber 
\end{equation}
Now we recall the decomposition $U=U_w\cdot \overline{U}_w$, which in particular respects measures. This allows us to partially unfold the $u$-integral. We arrive at
\begin{align}
	\int_{U(\mathbb{Q})\backslash U(\mathbb{A})} P(E,T^{\vee}\ast T; xy)\overline{\boldsymbol{\psi}(x)}dx &= \sum_{\substack{w\in W,\\ t\in Z(\mathbb{Q})\backslash T(\mathbb{Q})}} \int_{\overline{U}_w(\mathbb{Q})\backslash \overline{U}_w(\mathbb{A})}\int_{U_w(\mathbb{A})} F(tw\overline{u}uy)\boldsymbol{\psi}^{-1}(\overline{u}u)dud\overline{u} \nonumber \\
	&= \sum_{\substack{w\in W,\\ t\in Z(\mathbb{Q})\backslash T(\mathbb{Q})}} \delta(tw,\boldsymbol{\psi})\cdot \int_{U_w(\mathbb{A})} F(twuy)\boldsymbol{\psi}^{-1}(u)du, \nonumber
\end{align}
for 
\begin{equation}
	\delta(tw,\boldsymbol{\psi})  = \int_{\overline{U}_w(\mathbb{Q})\backslash\overline{U}_w(\mathbb{A})}\boldsymbol{\psi}(tw\overline{u}w^{-1}t^{-1}\overline{u}^{-1})d\overline{u}.\nonumber 
\end{equation}
Since ${\rm Vol}(\overline{U}_w(\mathbb{Q})\backslash\overline{U}_w(\mathbb{A}),d\overline{u})=1$ character orthogonality implies that $\delta(tw,\boldsymbol{\psi})\in \{0,1\}$. Further, we observe that the remaining $u$-integral factors. We write
\begin{equation}
	\int_{U_w(\mathbb{A})} F(twuy)\boldsymbol{\psi}^{-1}(u)du = \mathcal{O}_{\infty}^{\boldsymbol{\psi_{\infty}}}(tw, F_{\infty};y)\cdot \mathcal{O}_{\rm fin}^{\boldsymbol{\psi}}(tw) =  \mathcal{O}_{\infty}^{\boldsymbol{\psi_{\infty}}}(tw, F_{\infty};y)\cdot \prod_p\mathcal{O}_p^{\boldsymbol{\psi}_p}(tw,F_p), \label{orb}
\end{equation}
where the local orbital integrals were defined in \eqref{p_orb} and \eqref{arch_orb}. We arrive at:

\begin{prop}\label{prop:geo_exp}
We have
\begin{equation}
	I(E,T) = C_n^{-1}2^{n-1}\sum_{\substack{w\in W,\\ c\in \mathbb{Z}_{\geq 0}^{n-1},\\ v\in Z(\pm 1)\backslash V(\mathbb{R})}} \delta(vc^{\ast }w,\boldsymbol{\psi})\cdot \prod_p \mathcal{O}_p^{\boldsymbol{\psi}_{p}}(vc^{\ast}w,F_p)\cdot  \int_{\tilde{T}(\mathbb{R}_{>0})}\mathcal{O}_{\infty}^{\boldsymbol{\psi}_{\infty}}(vc^{\ast}w,F_{\infty};y)\cdot \overline{E({\rm y}(y))}d^{\ast}y
\end{equation}
Furthermore, $\delta(vc^{\ast}w,\boldsymbol{\psi})\neq 0$ implies that $(c,w)$ is admissible.
\end{prop}
\begin{proof}
Note that by the definition of the local components $F_p$ of $F$ we find that the product of the orbital integrals vanishes unless $t\in Z(\mathbb{Q})T_0(\mathbb{Q})$. Thus, if $t\in Z(\mathbb{Q})\backslash T(\mathbb{Q})$ contributes to the sum, then we can choose a representative modulo the center so that the $t = vc^{\ast}$ with $v\in Z(\pm 1)\backslash V(\mathbb{R})$ and $c\in \mathbb{Q}_{>0}^{n-1}$. Now we recall that above we have seen that $\mathcal{O}_p^{\boldsymbol{\psi}_p}(vc^{\ast}w,F_p) = 0$ unless $c\in \mathbb{Z}_p^{n-1}$. Since this is true for all $p$ we see that only $c\in \mathbb{Z}_{\geq 0}^{n-1}$ contribute. These observations lead to the geometric expression for $I(E,T)$ stated above.

Now suppose that $\delta(vc^{\ast}w,\boldsymbol{\psi})\neq 0$, This only happens when 
\begin{equation}
	\overline{u}\mapsto \boldsymbol{\psi}(tw\overline{u}w^{-1}t^{-1}\overline{u}^{-1}) \nonumber
\end{equation}
is trivial on $\overline{U}_{w}(\mathbb{Q})\backslash \overline{U}_w(\mathbb{A})$. By strong approximation this translates into the condition \cite[(4.1)]{AB} for trivial $N$ and $M$. The admissibility of $(c,w)$ follows immediately from the discussion in \cite[Section~4]{AB}.
\end{proof}

We are now ready to derive an important estimate for the geometric side.

\begin{lemma}\label{lm:est_geo}
For $1\ll Z$, we have
\begin{equation}
	I(E^{(Z,1,\ldots,1)},T_1) \ll Z^{2\eta_1+\epsilon}\mathcal{V}_l^{\epsilon}(\mathcal{V}_l + Z^{n-1}). \nonumber
\end{equation}
\end{lemma}
\begin{proof}
We start from the expansion in Proposition~\ref{prop:geo_exp}. Note that it suffices to bound
\begin{equation}
	S_{w,v} = \sum_{\substack{c\in \mathbb{Z}_{\geq 0}^{n-1},\\ (c,w) \text{ admissible}}}  \prod_p \mathcal{O}_p^{\boldsymbol{\psi}_{p}}(vc^{\ast}w,F_p)\cdot  \int_{\tilde{T}(\mathbb{R}_{>0})}\mathcal{O}_{\infty}^{\boldsymbol{\psi}_{\infty}}(vc^{\ast}w,F_{\infty};y)\cdot \overline{E({\rm y}(y))}d^{\ast}y.\nonumber
\end{equation}
In particular, it is enough to consider Weyl elements of the form $w=w_{d_1,\ldots,d_k}$. Treating the archimedean contribution using Lemma~\ref{arch_lem} we get
\begin{equation}
	S_{w,v} \ll Z^{2\eta_1+\epsilon} \sum_{\substack{c\in \mathbb{Z}_{\geq 0}^{n-1},\\ c_i\ll Z,\\  (c,w) \text{ admissible}}}  \prod_p \mathcal{O}_p^{\mathbf{1}}(vc^{\ast}w,W_p)\cdot (c_1\cdots c_{n-1})^{-1}.\label{after_first_estimate}
\end{equation}
Before we continue we recall the local zeta-functions $Z_w(F_p,\chi)$ defined in \eqref{def_Kloosterman_zeta}. Note that if $p\neq l$, then these are unramified and agree with the Kloosterman-set zeta-function studied in \cite{DR}. In particular, it is a consequence of \cite[Theorem~4.3]{DR} that the product $\prod_p Z_w(F_p,\chi)$ is absolutely convergent for sufficiently dominant $\chi$. 

Let's return to \eqref{after_first_estimate}. Since we are only interested in upper bounds we can smoothen the sharp cut-off $c_i\ll Z$ by inserting weights of the form $f(\frac{1}{Z}\cdot c)$, where $f\colon \mathbb{R}^{n-1}\to \mathbb{R}_{\geq 0}$ is a suitable smooth and compactly supported function. Next we recall that an admissible tuple $c$ is determined by its entries at the indices $n-d_1-\ldots-d_i$ for $i=1,\ldots,k-1$. Applying Mellin inversion in these variables leads to
\begin{equation}
	S_{w,v}\ll Z^{2\eta_1+\epsilon} \cdot \int_{\Re(\mathbf{s})=\boldsymbol{\sigma}}\prod_p Z_w(F_p,\chi_{\mathbf{s}}^{\rm ad})Z^{\mathbf{s}} [\mathfrak{M}f](\mathbf{s})d\mathbf{s},\nonumber
\end{equation}
where $\mathbf{s}=(s_1,\ldots,s_{k-1})$ has $k-1$ components, $[\mathfrak{M}f](\mathbf{s})$ is the multiple Mellin transform in the corresponding variables and $d\mathbf{s}=ds_1\cdots ds_{k-1}$. Recall that the characters $\chi_{\mathbf{s}}^{\rm ad}$ is taken to detect the admissibility condition. At this point we shift the contour to $\boldsymbol{\sigma} = \mathbf{1}+\boldsymbol{\epsilon}$ and estimate everything trivially. At the unramified places we do so using Lemma~\ref{triv_bound_KS}, while at the ramified place $p=l$ we apply Lemma~\ref{lm:another_lemm}. We are left with
\begin{equation}
	S_{w,v}\ll Z^{2\eta_1+k-1+\epsilon}\mathcal{V}_l\cdot \sup_{\mathbf{s}\in (\mathbf{1}+\boldsymbol{\epsilon})}\vert \mathcal{A}_{w}^{\chi^{\rm ad}_{\mathbf{s}}}\phi_1^{\chi^{\rm ad}_{\mathbf{s}}}(\mathbf{1}_n) \vert.\nonumber
\end{equation}
The desired estimate follows from Proposition~\ref{prop:crucial}.
\end{proof}

\section{The Density Estimate}

We are finally ready to prove the main theorem of this note:

\begin{theorem}\label{th_dens}
Let $M,\epsilon>0$ and let $l$ be prime. Then, for $\sigma\geq 0$ we have
\begin{equation}
	N_{\infty}(\sigma;l,M) \ll_{\epsilon,n} M^K\mathcal{V}_l^{1-\frac{2\sigma}{n-1}+\epsilon}.\nonumber
\end{equation}
\end{theorem}
\begin{proof}
We start by artificially inserting a parameter $Z\geq 1$:
\begin{equation}
	N_{\infty}(\sigma;l,M) \leq \sum_{\substack{\pi\mid X_l,\\ \Vert\mu_{\infty}(\pi)\Vert\leq M}}Z^{2\sigma_{\infty}(\pi)-2\sigma}.\nonumber
\end{equation}
We recall Lemma~\ref{low_bound_bessel}, Lemma~\ref{arch_test_fct}, Lemma~\ref{upp_bound_lpi} and Lemma~\ref{spec_ineq} to bound
\begin{multline}
	N_{\infty}(\sigma;l,M) \leq \sum_{\substack{\pi\mid X_l, \\ \Vert\mu_{\infty}(\pi)\Vert\leq M}}Z^{2\sigma_{\infty}(\pi)-2\sigma} \ll Z^{-2\sigma-2\eta_1}M^Kl^{\epsilon}\sum_{\pi\mid X_l}\frac{J_{\pi_l}(T_1)}{l(\pi)}\cdot \sum_{j=1}^r\vert\langle W_{\mu_{\infty}(\pi)},E_j^{(Z,1,\ldots,1)}\rangle_{\tilde{T}}\vert^2 \\
	\ll Z^{-2\sigma-2\eta_1}M^Kl^{\epsilon} \max_jI(E_j^{(Z,1,\ldots,1)},T_1)
\end{multline}
On the other hand, for $Z\leq M^{-K}l^{\frac{n}{2}}$ we can use Lemma~\ref{lm:est_geo} to bound
\begin{equation}
	I(E_j^{(Z,1,\ldots,1)},T_1) \ll Z^{2\eta_1}M^K\mathcal{V}_l^{1+\epsilon}.\nonumber
\end{equation}
Thus in total we have
\begin{equation}
	N_{\infty}(\sigma;l,M) \ll Z^{-2\sigma}M^K\mathcal{V}_l^{1+\epsilon}.\nonumber
\end{equation}
Choosing $Z = M^{-K}l^{\frac{n}{2}} = M^{-K}\mathcal{V}_l^{\frac{1}{n-1}}$ we obtain the result.
\end{proof}

A technical application of this density estimate is a bound for the distribution $J^S$ evaluated at suitable test functions. Before stating this result let us briefly introduce some notation. As in \cite[Section~2.7]{AB} (see also \cite[Section~8]{JK}) we choose a smooth non-negative function $h_0\colon \mathbb{R}_{\geq 0}\to [0,1]$ with support in $[0,2n]$ and $h_0(x) =1$ for $x\in [0,n]$. Then set $$h_T(g) = \frac{1}{\Vert h_0(\frac{\Vert \cdot \Vert}{T})\Vert_{L_1}}h_0\left(\frac{\Vert g\Vert}{T}\right), \text{ for }g\in \textrm{SL}_n(\mathbb{R}).$$ This defines a nice element in the (spherical) archimedean Hecke algebra. The finite places are treated at once by setting
\begin{equation}
	h_l = \frac{1}{{\rm Vol}(I(l))}\mathbbm{1}_{I(l)}.\nonumber
\end{equation}
Globally, we define $h_{T,l}\in \mathcal{C}_c^{\infty}(\textrm{GL}_n(\mathbb{A})^1)$ by
\begin{equation}
	h_{T,l}(g_{\infty}g_{\rm fin}) = [h_T\ast h_T^{\vee}](g_{\infty})\cdot h_l(g_{\rm fin}).\label{eq:def_glob_test}
\end{equation}

Given a test function, for example $h_{T,l}$ as defined above, Arthur \cite{Ar2} associates a distribution $J_{\chi}^S(h_{T,l})$, where $\chi$ is a cuspidal data attached to some Levi subgroup $M$ of $G$ and $S$ is a truncation parameter. For the sake of brevity we ignore the precise definition and refer to \cite[Proposition~6.1]{JK} for a working expression instead. We set
\begin{equation}
	J^S(h_{T,l}) = \sum_{\chi}J_{\chi}^S(H_{T,l}), \nonumber
\end{equation}
where the sum runs over all equivalence classes of cuspidal data and is absolutely convergent. See \cite[Proposition~6.3]{JK} for a convenient expression (for sufficiently dominant truncation parameter $S$).

\begin{prop}\label{orb_bounds}
In the notation as above we have
\begin{equation}
	J^S(h_{T,l}) \ll \Vert S\Vert^{n-1}(Tl)^{\epsilon} (1+T^{-n(n-1)}\mathcal{V}_l),\nonumber
\end{equation}
for any $S\in\mathfrak{a}_{M_0}$.
\end{prop}
This is a version of \cite[Proposition~8.1]{JK}. The proof is also a straight forward modification of the proof given in \cite[Section~8.4]{JK}. We only sketch the argument.
\begin{proof}
We proceed as in \cite[Section~8.4]{JK} and estimate $J^T(h_{T,l})$ using the expression given in \cite[Proposition~6.3]{JK}. This is done using \cite[Proposition~5.5]{JK}, which works for general open compact subgroups of $K_{\rm fin}$. Noting that the level of $I(l)$ is simply $l$ we arrive at
\begin{equation}
	J^S(h_{T,l}) \ll \Vert S\Vert^{n-1}(Tl)^{\epsilon}T^{-n(n-1)}\sum_{P\supset B}\sum_{\substack{\pi\in \Pi(M_P(\mathbb{A})),\\ \Vert \mu_{\pi}\Vert \ll (Tl)^{\epsilon}}}\dim\left(\mathcal{A}_{\pi}^2(P)^{K_{\infty}^0I(l)}\right)T^{2n\Vert \Re(\mu_{\pi})\Vert} + O(1).\nonumber
\end{equation} 
Unknown notation is to be understood as in \cite{JK}. The remaining sum is estimated by injecting the density estimate Theorem~\ref{th_dens} above. Closely following the discussion below \cite[Lemma~8.10]{JK} leads us to the bound %Note however that in contrast to the principal congruence subgroup the local multiplicities $\dim_{\mathbb{C}}\pi_l^{I_l(l)}$ are uniformly bounded (by a constant depending only on $n$). 
\begin{equation}
	\sum_{P\supset B}\sum_{\substack{\pi\in \Pi(M_P(\mathbb{A})),\\ \Vert \mu_{\pi}\Vert \ll (Tl)^{\epsilon}}}\dim\left(\mathcal{A}_{\pi}^2(P)^{K_{\infty}^0I(l)}\right)T^{2n\Vert \Re(\mu_{\pi})\Vert} \ll (Tl)^{\epsilon}(T^{n(n-1)}+\mathcal{V}_l),\nonumber
\end{equation}
which is exactly needed to complete the proof.
\end{proof}

\section{Applications}

As an application we can generalize the optimal lifting result for the action on flags in \cite[Section~5]{KL} from ${\rm GL}_3$ to ${\rm GL}_n$. To do so we also use recent results from \cite{JK}.

Let us start by introducing the problem. For a prime $l$ let $\mathbbm{1}_l$ be the finite field with $l$ elements and let $\mathfrak{B}_l$ denote the set of complete flags in $\mathbb{F}_l^n$. That is
\begin{equation}
	\mathfrak{B}_l = \{(V_1,\ldots,V_{n-1})\colon 0<V_1<\ldots<V_{n-1}<\mathbb{F}_l^n\}.\nonumber
\end{equation}
The action $\Phi_l\colon {\rm SL_n}(\mathbb{Z})\to {\rm Sym}(\mathfrak{B}_l)$ gives rise to the congruence subgroup
\begin{equation}
	\Gamma_2'(l) = \{ \gamma\in {\rm SL}_n(\mathbb{Z})\colon \Phi_l(q)(\gamma)(\mathbf{1}) = \mathbf{1}, \nonumber
\end{equation}
where $$\mathbf{1} =\left( \langle e_n\rangle, \langle e_n,e_{n-1}\rangle,\ldots, \langle e_n,\ldots, e_2\rangle\right)$$ is the standard flag. In coordinates we have
\begin{equation}
	\Gamma_2'(l) = \{ \gamma\in {\rm SL}_n(\mathbb{Z})\colon \gamma^{\top} \in B(\mathbb{F}_l) \text{ mod }l\}.
\end{equation}
Note that $\Gamma_2(l)^{\top} = \Gamma_2'(l)$ is the lattice we worked with above. The conjugation appears for technical reasons.

We can prove the following counting result:

\begin{theorem}\label{th:counting}
There exists a constant $C=C(n)>0$ such that for every prime $l$ and $\epsilon>0$ we have
\begin{equation}
	\sharp\{(\gamma,x)\in {\rm SL}_n(\mathbb{Z})\times \mathfrak{B}_l\colon \Vert \gamma\Vert_{\infty}\leq T,\, \Phi_l(\gamma)(x)=x  \}\ll_{\epsilon,n} (Tl)^{\epsilon}[T^{n(n-1)}+l^{\frac{n(n-1)}{2}}T^{\frac{n(n-1)}{2}}].\nonumber
\end{equation}
\end{theorem}   

This can be compared to \cite[Theorem~5.2]{KL}. Note that the proof in loc. cit is a direct counting argument. To make this easier the archimedean constraint $\Vert\gamma\Vert\leq T$ is replaced by the stronger condition $\Vert \gamma\Vert_{\infty} \cdot\Vert \gamma^{-1}\Vert_{\infty}\leq T$, which changes the exponents in $T$. Here we give an spectral argument based on ideas from \cite{JK}:

\begin{proof}[Proof of Theorem~\ref{th:counting}]
Recall the global test function $h_{T,l}$ defined in \eqref{eq:def_glob_test}. Note that the archimedean component is $h_T\ast h_T^{\vee}$ where $h_T$ agrees with the test function from  \cite[Section~2.7]{AB} (see also \cite[Section~8]{JK}). We record the upper bound
\begin{equation}
	\sharp\{(\gamma,x)\in {\rm SL}_n(\mathbb{Z})\times \mathfrak{B}_l\colon \Vert \gamma\Vert_{\infty}\leq T,\, \Phi_l(\gamma)(x)=x  \} \ll {\rm Vol}(B_T)\cdot \sum_{x\in \Gamma_2(l)\backslash {\rm SL}_n(\mathbb{Z})}\sum_{\gamma\in \Gamma_2(l)}[h_{C_1\sqrt{T}}\ast h_{C_1\sqrt{T}}](x^{-1}\gamma x).\label{eq:first_automorphisation}
\end{equation}
Here we have used the inequality $\mathbbm{1}_{B_{T}} \ll 	\mathbbm{1}_{B_{C_1\sqrt{T}}} \ast \mathbbm{1}_{B_{C_1\sqrt{T}}}$, for some constant $C_1>0$, given on \cite[p.38]{JK} and inserted the smooth weight using $\mathbbm{1}_{B_{C_1\sqrt{T}}} \ll h_{C_1\sqrt{T}}$.

Next we define the automorphic kernel
\begin{equation}
	K_{h_{T,l}}(x,y) = \sum_{\gamma\in {\rm GL}_n(\mathbb{Q})}h_{T,l}(x^{-1}\gamma y), \text{ for }x,y\in \textrm{GL}_n(\mathbb{A}).\nonumber
\end{equation}
We write $\iota_{\infty}$ for the embedding of $\textrm{GL}_n(\mathbb{R})$ in $\textrm{GL}_n(\mathbb{A})$ at the archimedean place. Using \eqref{eq:intersection} we see that
\begin{equation}
	K_{h_{T,l}}(\iota_{\infty}(x),\iota_{\infty}(x)) = \frac{1}{\textrm{Vol}(I(l))}\sum_{\gamma\in \Gamma_2(l)}[h_T\ast h_T^{\vee}](x^{-1}\gamma x), \text{ for }x\in \textrm{GL}_n(\mathbb{R}). \nonumber
\end{equation}
After an application of \cite[Lemma~8.7]{JK} this expression for $K_{h_{T,l}}(\iota_{\infty}(x),\iota_{\infty}(x))$ can be inserted in \eqref{eq:first_automorphisation}. We obtain
\begin{multline}
	\sharp\{(\gamma,x)\in {\rm SL}_n(\mathbb{Z})\times \mathfrak{B}_l\colon \Vert \gamma\Vert_{\infty}\leq T,\, \Phi_l(\gamma)(x)=x  \} \\ \ll {\rm Vol}(B_T\times I(l))\cdot \sum_{x\in \Gamma_2(l)\backslash {\rm SL}_n(\mathbb{Z})}K_{h_{C_1\sqrt{T},l}}(\iota_{\infty}(x),\iota_{\infty}(x)),\label{eq:some_ineq}
\end{multline}	
Our goal is now to insert a small additional average in the expression above. This can be done using the arguments from the proof of \cite[Proposition~8.6]{JK}. We first use automorphy and strong approximation to write
\begin{equation}
	\sum_{x\in \Gamma_2(l)\backslash {\rm SL}_n(\mathbb{Z})}K_{h_{C_1\sqrt{T},l}}(\iota_{\infty}(x),\iota_{\infty}(x)) = \sum_{x_{\rm fin}\in K_{\rm fin}/I(l)} K_{h_{C_1\sqrt{T},l}}(x_{\rm fin},x_{\rm fin}).\nonumber
\end{equation}
Next, after replacing $C_1$ by another constant $C_2>0$ we can insert a small integral at the archimedean place. By $I(l)$-invariance we can also add the corresponding integral at the finite places. We arrive at
\begin{equation}
	\sum_{x_{\rm fin}\in K_{\rm fin}/I(l)} K_{h_{C_1\sqrt{T},l}}(x_{\rm fin},x_{\rm fin}) \ll \frac{1}{{\rm Vol}(I(l))}\sum_{x_{\rm fin}\in K_{\rm fin}/I(l)}\int_{{\rm GL}_n(\mathbb{Q})\backslash [B_{R_0}\times I(l)]}K_{h_{C_2\sqrt{T},l}}(x_{\rm fin}y,x_{\rm fin}y)dy,\nonumber
\end{equation}
for a fixed $R_0>n$. Combining sum and integral and inserting the resulting inequality in \eqref{eq:some_ineq} leads to
\begin{equation}
	\sharp\{(\gamma,x)\in {\rm SL}_n(\mathbb{Z})\times \mathfrak{B}_l\colon \Vert \gamma\Vert_{\infty}\leq T,\, \Phi_l(\gamma)(x)=x  \} \ll {\rm Vol}(B_T)\cdot\int_{\Omega} K_{h_{C_2\sqrt{T},l}}(x,x)dx,\nonumber
\end{equation}
where $\Omega={\rm GL}_n(\mathbb{Q})\backslash [B_{R_0}\times K_{\rm fin}]$. Recall, for example from \cite{Ma}, that 
\begin{equation}
	{\rm Vol}(B_T)\ll T^{n(n-1)}.\nonumber
\end{equation}
Thus mimicking the argument from the proof of \cite[Proposition~8.3]{JK} we get
\begin{equation}
	\sharp\{(\gamma,x)\in {\rm SL}_n(\mathbb{Z})\times \mathfrak{B}_l\colon \Vert \gamma\Vert_{\infty}\leq T,\, \Phi_l(\gamma)(x)=x  \} \ll T^{n(n-1)}\cdot J^S(h_{C_2\sqrt{T},l}) \nonumber
\end{equation}
for some $d(S)\ll 1+\log(T)$. The desired estimate follows from the technical estimate of $J^S(h_{C_2\sqrt{T},l})$ given in Proposition~\ref{orb_bounds} above. Indeed, we get
\begin{equation}
	\sharp\{(\gamma,x)\in {\rm SL}_n(\mathbb{Z})\times \mathfrak{B}_l\colon \Vert \gamma\Vert_{\infty}\leq T,\, \Phi_l(\gamma)(x)=x  \} \ll (Tl)^{\epsilon}[T^{n(n-1)}+l^{\frac{n(n-1)}{2}}T^{\frac{n(n-1)}{2}}].\nonumber
\end{equation}
\end{proof}

Similarly one also obtains the following lifting theorem:

\begin{theorem}\label{th:lifting}
Let $l$ be a prime and let $\mathfrak{B}_l$ and $\Phi_l$ be as above. Then for every $\epsilon>0$, as $l\to\infty$, there exists a set $Y\subset \mathfrak{B}_l$ of size $\sharp Y\geq (1-o_{\epsilon}(1))\sharp \mathfrak{B}_l$, such that for every $x\in Y$ there exists a set $Z_x\subset \mathfrak{B}_l$ of size $\sharp Z_x\geq  (1-o_{\epsilon}(1))\sharp \mathfrak{B}_l$, such that for every $y\in Z_x$ there exists $\gamma\in {\rm SL}_n(\mathbb{Z})$ satisfying $\Vert \gamma\Vert_{\infty} \leq l^{\frac{1}{2}+\epsilon}$, such that $\Phi_l(\gamma)x=y$.
\end{theorem}

This generalizes \cite[Theorem~5.1]{KL} from ${\rm SL}_3$ to ${\rm SL}_n$. Note that in loc. cit. the result is derived from (a version of) Theorem~\ref{th:counting} directly using the spectral gap. This is possible in general. On the other hand, one can also give a direct spectral proof following \cite{JK} (and \cite{AB}). We omit the details.


\begin{thebibliography}{BHM}
	
	
\bibitem[AMSS]{AMSS} P. Aluffi, L. C. Mihalcea, J. Sch{\"u}rmann, C. Su, \emph{Motivic Chern classes of Schubert cells, Hecke algebras, and applications to Casselman's problem}, Ann. Sci. \'Ec. Norm. Sup\'er. (4) 57 (2024), no. 1, 87--141. 

\bibitem[A]{A} E. Assing, \emph{A note on Sarnak's density hypothesis for $Sp(4)$}, Comment. Math. Helv. (2026), published online first. 

\bibitem[AB]{AB} E. Assing, V. Blomer, \emph{The density conjecture for principal congruence subgroups}, Duke Math. J. 173 (2024), no. 7, 1359--1426.
	
\bibitem[ABN]{ABN} E. Assing, V. Blomer, P. D. Nelson \emph{Local analysis of the Kuznetsov formula and the density conjecture}, {\tt arXiv:2404.05561}

\bibitem[Ar1]{Ar2} J. Arthur, \emph{An introduction to the trace formula}, in: Harmonic analysis, the trace formula, and {S}himura varieties, Clay Math. Proc. \textbf{4} (2005), 1-263.

\bibitem[BaM]{BaM} E. M. Baruch, Z. Mao, \textit{Bessel identities in the Waldspurger correspondence over the real numbers}, Israel J. Math. \textbf{145} (2005), 1--81. 

\bibitem[Bl]{Bl} V. Blomer, \emph{Density theorems for {\rm GL(n)}}, Invent. Math. 232 (2023), no. 2, 683--711.

\bibitem[BB]{BB} V. Blomer, F. Brumley, \emph{The role of the Ramanujan conjecture in analytic number theory}, Bull. Amer. Math. Soc.
\textbf{50} (2013), 267-320

\bibitem[BBM]{BBM} V. Blomer, J. Buttcane, P. Maga, \emph{Applications of the Kuznetsov formula on ${\rm GL}(3)$ {II}: the level aspect}, Math. Ann. \textbf{369} (2017), no. 1-2, 723--759. 

\bibitem[BBR]{BBR} V. Blomer, J. Buttcane, N. Raulf, \emph{A Sato-Tate law for ${\rm GL}(3)$}, Comment. Math. Helv. \textbf{89} (2014), no. 4, 895--919.


\bibitem[BrM]{BrM} F. Brumley, D. Mili\'cevi\'c, \emph{Counting cusp forms by analytic conductor}, Ann. Sci. \'Ec. Norm. Sup\'er. (4) 57 (2024), no. 5, 1473--1597. 
	

\bibitem[BlM]{BlM} V. Blomer, S.~H. Man, \emph{Bounds for Kloosterman sums on GL(n)},  Math. Ann. 390 (2024), no. 1, 1171--1200. 



\bibitem[BN1]{BM1} D. Bump, M. Nakasuji, \emph{Casselman's basis of {I}wahori vectors and the {B}ruhat order}, Canad. J. Math.\textbf{63} (2011), 1238--1253


\bibitem[BN2]{BM2} D. Bump, M. Nakasuji, \emph{Casselman's basis of {I}wahori vectors and {K}azhdan-{L}usztig polynomials}, Canad. J. Math. \textbf{71} (2019), 1351--1366


\bibitem[Bo]{Bo} A. Borel,\emph{Admissible representations of a semi-simple group over a local field with vectors fixed under an {I}wahori subgroup},Invent. Math. \textbf{35} (1976), 233--259

\bibitem[Ca]{Ca} W. Casselman,\emph{The unramified principal series of {${p}$}-adic groups. {I}. {T}he spherical function}, Compositio Math. \textbf{40} (1989), 387--406



\bibitem[CS]{CSh} W. Casselman, J. Shalika,\emph{The unramified principal series of {$p$}-adic groups. {II}.	{T}he {W}hittaker function}, Compositio Math.\textbf{41} (1980), 207--231



\bibitem[Do]{Do} H. Donnelly, \emph{On the cuspidal spectrum for finite volume symmetric spaces}, J. Differential Geometry \textbf{17} (1982), no. 2, 239--253.

\bibitem[DR]{DR} R. Dabrowski, M. Reeder, \emph{Kloosterman sets in reductive groups}, J. Number Theory \textbf{73} (1998), 228-255


\bibitem[GK]{GK} K. Golubev, A. Kamber, \emph{On Sarnak's density conjecture and its applications}, Forum Math. Sigma 11 (2023), Paper No. e48, 51 pp.
	


\bibitem[Ja]{Ja} S. Jana, \emph{Applications of analytic newvectors for $\text{GL}(n)$}, Math. Ann. \textbf{380} (2021), no. 3-4, 915--952. 

\bibitem[JK]{JK} S. Jana, A. Kamber, \textit{On the local $L^2$-Bound of the Eisenstein Series}, Forum Math. Sigma 12 (2024), Paper No. e76, 44 pp.
	
\bibitem[KL]{KL} A. Kamber, H. Lavner, \emph{Optimal Lifting for the Projective Action of $SL_3(Z)$}, Algebra Number Theory 17 (2023), no. 3, 749--774.

%\bibitem[Kn]{Kn} A. W. Knapp, \emph{The Gindikin-Karpelevic formula and intertwining operators}, in Gindikin, S. G. (ed.), Lie groups and symmetric spaces. In memory of F. I. Karpelevich, Amer. Math. Soc. Transl. Ser. 2, vol. 210, Providence, R.I.: American Mathematical Society, pp. 145–159


\bibitem[Hum]{Hum} P. Humphries, \emph{Density Theorems for Exceptional Eigenvalues for Congruence Subgroups}, Algebra \& Number
Theory \textbf{12} (2018), 1581-1610.


\bibitem[Hux]{Hux} M. Huxley, \emph{Exceptional eigenvalues and congruence subgroups}, in: The Selberg trace formula and
related topics, Contemp. Math. \textbf{53} (1986), 341-349

\bibitem[Iw]{Iw} H. Iwaniec, \emph{Small eigenvalues of Laplacian for $\Gamma_0(N)$}, Acta Arith. \textbf{56} (1990), 65-82

\bibitem[La]{Lan} R. P. Langlands, \emph{Euler products}, Yale Mathematical Monographs, 1. Yale University Press, New Haven, Conn.-London, 1971.

\bibitem[Li]{Li} X. Li, \emph{Upper bounds on $L$-functions at the edge of the critical strip}, IMRN 2010, 727-755

\bibitem[Lin]{Lin} J. Linn, \emph{Bounds for Kloosterman Sums for $GL(n)$}, {\tt arXiv:2412.04976}.

\bibitem[LM]{LM} E. Lapid, Z. Mao, \emph{A conjecture on {W}hittaker-{F}ourier coefficients of cusp
	forms}, J. Number Theory \textbf{146} (2015), 448-505

\bibitem[Mac]{Ma} F. Maucourant, \emph{Homogeneous asymptotic limits of {H}aar measures of semisimple linear groups and their lattices}, Duke Math. J. \textbf{136}  (2007),357--399

\bibitem[Man]{Man} S. H. Man, \emph{A density theorem for ${\rm Sp}(4)$}, J. Lond. Math. Soc. (2) \textbf{105} (2022), no. 4, 2047--2075.

\bibitem[Mia]{Mi} X. Miao, \emph{Bessel Functions and Kloosterman Integrals on ${\rm GL}(n)$},  J. Funct. Anal. 286 (2024), no. 4, Paper No. 110267, 62 pp.

\bibitem[MT]{MT} J. Matz, N. Templier, \emph{Sato-Tate equidistribution for families of Hecke-Maass forms on ${\rm SL}(n, \Bbb R )\slash {\rm SO}(n)$}, Algebra Number Theory \textbf{15} (2021), no. 6, 1343--1428.	


\bibitem[MW]{MW2} C. M{\oe}glin, J.-L. Waldspurger, \emph{Le spectre r{\'e}siduel de ${\rm GL}(n)$}, Ann. Sci {\'E}cole Norm. Sup. \textbf{ 22 } (1989), 605--674.


\bibitem[Re1]{Re1} M. Reeder, \emph{On certain Iwahori invariants in the unramified principal series}, Pacific J. Math. \textbf{153} (1992), no. 2, 313--342.

\bibitem[Re2]{Re2} M. Reeder, \emph{{$p$}-adic Whittaker functions and vector bundles on flag manifolds}, Compositio Math. \textbf{85} (1993), no. 1, 9--36.


\bibitem[Sa1]{Sa1} P. Sarnak, \emph{Diophantine Problems and Linear Groups}, Proceedings of the ICM Kyoto (1990), 459-471

\bibitem[Sa2]{Sa4} P. Sarnak, \emph{Notes on the generalized Ramanujan conjectures}, in: Harmonic analysis, the trace formula, and 
Shimura varieties, Clay Math. Proc. \textbf{4} (2005), 659-685


\bibitem[St]{Ste} G. Stevens, \emph{Poincaré series on ${\rm GL}(r)$ and Kloostermann sums}, Math. Ann. \textbf{277} (1987), no. 1, 25--51. 


\bibitem[Wa]{Wa} N. R. Wallach, \emph{Real reductive groups. II}, Pure and Applied Mathematics, 132-II. Academic Press, Inc., Boston, MA, 1992.

\bibitem[Ze]{Ze} A. V. Zelevinsky, \emph{Induced representations of reductive {${p}$}-adic groups. {II}. {O}n irreducible representations of {${\rm GL}(n)$}}, Ann. Sci. \'{E}cole Norm. Sup. (4) \textbf{13} (1980), 165--210

\end{thebibliography}
\end{document}